\documentclass[final,times,1p,leqno,11pt]{elsarticle}
\makeatletter
\def\@author#1{\g@addto@macro\elsauthors{\normalsize%
    \def\baselinestretch{1}%
    \upshape\authorsep#1\unskip\textsuperscript{%
      \ifx\@fnmark\@empty\else\unskip\sep\@fnmark\let\sep=,\fi
      \ifx\@corref\@empty\else\unskip\sep\@corref\let\sep=,\fi
      }%
    \def\authorsep{\space and\space}%
    \global\let\@fnmark\@empty
    \global\let\@corref\@empty
    \global\let\sep\@empty}%
    \@eadauthor={#1}
}

\usepackage[latin1]{inputenc}
\usepackage{enumerate}
\usepackage{amsfonts}
\usepackage{amsmath}
\usepackage{amsthm}
\usepackage{amssymb}
\usepackage{latexsym}
\usepackage{multicol}
\usepackage{verbatim}
\usepackage{amscd}
\usepackage{mathrsfs}
\usepackage[usenames]{color}
\usepackage[all]{xy}

\makeatletter
\def\ps@pprintTitle{%
     \let\@oddhead\@empty
     \let\@evenhead\@empty
     \def\@oddfoot{\footnotesize\itshape%
       }%
     \let\@evenfoot\@oddfoot}
\makeatother
\journal{\relax}

\numberwithin{equation}{section}
\theoremstyle{definition}
\newtheorem{thm}{Theorem}[section]
\newtheorem{prop}[thm]{Proposition}
\newtheorem{cor}[thm]{Corollary}
\newtheorem{lem}[thm]{Lemma}
\newtheorem{defn}[thm]{Definition}

\newtheorem{rem}[thm]{Remark}
\newtheorem{conj}[thm]{Conjecture}

\newcounter{cnt}
\newenvironment{enumerit}{\begin{list}{{\hfill\rm(\roman{cnt})\hfill}}{%
\settowidth{\labelwidth}{{\rm(iv)}}\leftmargin=\labelwidth%
\advance\leftmargin by \labelsep\rightmargin=0pt\usecounter{cnt}}}{\end{list}}

\newcommand{\wt}{\mathrm{wt}}

\def\cal#1{\text{$\mathcal{#1}$}}

\def\lie#1{\text{$\mathfrak{#1}$}}
\def\tlie#1{\tilde{\mathfrak{#1}}}

\newcommand{\ev}{\mathrm{ev}}

\def\tprod_#1^#2{\text{\scriptsize$\prod\limits_{\text{\footnotesize$#1$}}^{\text{\footnotesize$#2$}}$}}

\def\gb#1{{\mbox{\boldmath $#1$}}}
\def\opl_#1^#2{\text{\scriptsize$\bigoplus\limits_{\text{\footnotesize$#1$}}^{\text{\footnotesize$#2$}}$}}
\def\otm_#1^#2{\text{\scriptsize$\bigotimes\limits_{\text{\footnotesize$#1$}}^{\text{\footnotesize$#2$}}$}}
\def\tsum_#1^#2{\text{\scriptsize$\sum\limits_{\text{\footnotesize$#1$}}^{\text{\footnotesize$#2$}}$}}
\def\tbinom#1#2{\text{$\left(\begin{smallmatrix} #1\\#2\end{smallmatrix}\right)$}}
\def\ch{{\rm ch}}

\begin{document}
\begin{frontmatter}
\title{Finite-dimensional representations of twisted hyper loop algebras.\tnoteref{fn1}}
\author{Angelo Bianchi}
\emailauthor{angelo@ime.unicamp.br}{A.~B.}
\author{Adriano Moura}\emailauthor{aamoura@ime.unicamp.br}{A.~M.}
\tnotetext[fn1]{Partially supported by the FAPESP grant 2007/07456-9 (A.~B.) and  the CNPq grant 306678/2008-0 (A.~M.)}
\address{Universidade Estadual de Campinas}

\begin{abstract}
We investigate the category of finite-dimensional representations of twisted hyper loop algebras, i.e., the hyperalgebras associated to twisted loop algebras over finite-dimensional simple Lie algebras. The main results are the classification of the irreducible modules, the definition of the universal highest-weight modules, called the Weyl modules, and, under a certain mild restriction on the characteristic of the ground field, a proof that the simple modules and the Weyl modules for the twisted hyper loop algebras are isomorphic to appropriate simple and Weyl modules for the non-twisted hyper loop algebras, respectively, via restriction of the action.
\end{abstract}

\end{frontmatter}

\section*{Introduction}

In the last years, there has been an intense study of the finite-dimensional representation theory of algebras of the form $(\lie g\otimes A)^{G}$, where $\lie g$ is a finite-dimensional simple Lie algebra over the complex numbers, $A$ is an associative algebra and $G$ is a group acting nontrivially on $\lie g\otimes A$  (see \cite{CFS,fkks,fk,fms,ns,nss,S2} and references therein). These algebras can be regarded as generalizations of the (centerless) twisted affine Kac-Moody algebras, which are recovered in the particular case that $A=\mathbb C[t,t^{-1}]$ and $G$ is generated by a Dynkin diagram automorphism $\sigma$ of $\lie g$. On the other hand, it was initiated in \cite{hyperlar} the study of finite-dimensional representations of hyper loop algebras in positive characteristic, which is morally equivalent to studying representations of the associate affine Kac-Moody group (see the introduction of \cite{hyperlar} for more details).
The goal of this paper is to establish the basic results about the finite-dimensional representations of twisted hyper loop algebras, such as the classification of the simple modules, the development of the underlying highest-weight theory, and the existence of universal highest-weight objects, called the Weyl modules.

The hyperalgebras are constructed by considering an integral form of the corresponding universal enveloping algebra and then tensoring this form over $\mathbb Z$ with an arbitrary field $\mathbb F$, which we shall assume is algebraically closed. In the case of $\lie g$, one considers Konstant's integral form of $U(\lie g)$ and the corresponding hyperalgebra is denoted by $U_\mathbb F(\lie g)$. The affine analogues of Kostant's form were obtained by Garland \cite{G} in the non-twisted case and by Mitzman \cite{mitz} in the twisted case. The corresponding hyperalgebras are denoted by  $U_\mathbb F(\tlie g)$ and $U_\mathbb F(\tlie g^\sigma)$, respectively. The details of the construction of these algebras are described in Section \ref{s:algs}. Although this section is mostly dedicated to reviewing these constructions and fixing the basic notation of the paper, a few extra technical details required proof. Such details, which are proved in Section \ref{s:evmaps}, allow us to regard $U_\mathbb F(\tlie g^\sigma)$ as a subalgebra of $U_\mathbb F(\tlie g)$, a fact not as immediate as in the characteristic zero setting -- for instance, Mitzman's integral form is not a subalgebra of Garland's integral form for types $A_{2n}^{(2)}$ and $D_4^{(3)}$. In fact, for type $A_{2n}^{(2)}$, we need to suppose that the characteristic of $\mathbb F$ is different from $2$ because, otherwise, the hyperalgebra of the simple Lie algebra $\lie g_0:=\lie g^\sigma$ is not a subalgebra of $U_\mathbb F(\tlie g^\sigma)$ (see Remark \ref{r:KinGM}). Once the inclusion $U_\mathbb F(\tlie g^\sigma)\hookrightarrow U_\mathbb F(\tlie g)$ is established, it is possible to consider evaluation maps $U_\mathbb F(\tlie g^\sigma)\to U_\mathbb F(\lie g)$, which are important for the development of the representation theory.

Section \ref{s:gandtg} is dedicated to reviewing the relevant facts about the finite-dimensional representation theories of $U_\mathbb F(\lie g)$ and $U_\mathbb F(\tlie g)$. Our main results are in Sections \ref{s:main} and \ref{s:tmvr}.

In Section, \ref{s:main} we obtain the basic results of the theory mentioned at the end of the first paragraph of this introduction. We begin by developing the relevant weight theory, called $\ell$-weight theory, which will lead to the notion of Drinfeld polynomials. Still in Section \ref{ss:tlw}, we obtain a set of relations satisfied by any highest-$\ell$-weight vector generating a finite-dimensional module (Proposition \ref{ellhwreltw}). Then, in Section  \ref{ss:evm}, we prove that every simple module must be a highest-$\ell$-weight module with a particular kind of highest $\ell$-weight (Drinfeld polynomials), construct the evaluation modules, and compute the Drinfeld polynomials of the simple ones. This eventually leads to the classification of the finite-dimensional simple $U_\mathbb F(\tlie g^\sigma)$-modules (Theorem \ref{t:csm}).
Next, in Section \ref{ss:wm}, we define the Weyl modules by generators and relations and prove that they are finite-dimensional (Theorem \ref{t:twfd}). It then follows from Proposition \ref{ellhwreltw} that they are the universal finite-dimensional highest-$\ell$-weight modules. No further restriction on $\mathbb F$ are required in the development of this section. It is worth remarking that, once this results are established, one can pass to the case of non-algebraically closed ground fields by using the approach of \cite{JM2}. In fact, the same proofs can be used {\it mutatis-mutandis} and the same is true for the results of \cite{JM3}.

In the last section, we prove the positive characteristic analogue of a result of \cite{CFS} stating that the simple modules and the Weyl modules for $U_\mathbb F(\tlie g^\sigma)$ can be obtained by restriction of the action from appropriate simple and Weyl modules for $U_\mathbb F(\tlie g)$, respectively (Theorem \ref{t:tmvr}). However, we need to suppose that the characteristic of $\mathbb F$ is not equal to the order of the Dynkin diagram automorphism $\sigma$.
Moreover, the proof of Theorem \ref{t:tmvr} in the case of Weyl modules is the only proof of the paper which could not be performed intrinsically in the context of hyperalgebras. Instead, we had to use the characteristic zero version of the result and apply the technique of reduction modulo $p$. In particular, this part of the proof depends on a conjecture on the independence of the dimension of the Weyl modules on the ground field (see Remark \ref{r:conj} for comments on the proof of this conjecture).

\noindent
\textbf{Acknowledgments:} The results of this paper were obtained during the Ph.D. studies of the first author under the supervision of the second author. We thank A. Engler and P. Brumatti for useful discussions and helpful references about Henselian discrete valuation rings.

\section{The algebras}\label{s:algs}

Throughout this work, $\mathbb{C}$ (respectively $\mathbb{C}^\times$) denotes the set of complex (respectively nonzero complex) numbers and $\mathbb{Z}$ (respectively $\mathbb{Z}_+$ and $\mathbb N$) is the set of integers (respectively non-negative and positive integers).

\subsection{Simple Lie algebras and diagram automorphisms}\label{s:simple}

Let $I$ be the set of vertices of a finite-type and connected Dynkin diagram and let $\lie g$ be the associated finite-dimensional simple Lie algebra over $\mathbb C$. Fix a triangular decomposition $\lie g=\lie n^-\oplus\lie h\oplus\lie n^+$ and denote by $R$ and $R^+$ the associated root system and set of positive roots. Let $\{\alpha_i:{i\in I}\}$ (resp. $\{ \omega_i :{i\in I}\}$) denote the simple roots (resp. fundamental weights) and set $Q=\oplus_{i\in I} \mathbb Z \alpha_i, Q^+=\oplus_{i\in I} \mathbb Z_+ \alpha_i, P=\oplus_{i\in I} \mathbb Z \omega_i$, and $P^+=\oplus_{i\in I} \mathbb Z_+ \omega_i$.  Let $\lie g_\alpha \ (\alpha \in R)$ denote the associated root space  and let $\theta$ be the highest root of $\lie g$.

Fix a diagram automorphism $\sigma$ of $\lie g$ and let $m\in\{1,2,3\}$ be its order. It induces a permutation of $R$ given by $$\displaystyle{\sigma\left(\tsum_{i \in I}^{} n_i \alpha_i\right)=\tsum_{i \in I}^{} n_i \alpha_{\sigma(i)}}$$
and a unique Lie algebra automorphism of $\lie g$ defined by $\sigma(x^\pm_i)=x^\pm_{\sigma(i)}$ for all $i\in I$. For $\alpha\in R^+$, set
\begin{equation*}
\Gamma_\alpha = \#\ \{\sigma^j(\alpha):j\in\mathbb Z\}
\end{equation*}
and notice that $\Gamma_\alpha\in\{1,m\}$. Observe also that $\sigma(\lie g_\alpha)=\lie g_{\sigma(\alpha)}$ for all $\alpha\in R$.
If $\zeta$ is a primitive $m^{th}$ root of unity and $\lie g_\epsilon=\{x\in\lie g: \sigma(x)=\zeta^\epsilon x\}$, then
\begin{equation*}
\lie g=\opl_{\epsilon=0}^{m-1}\lie g_\epsilon \quad\text{and}\quad [\lie g_\epsilon, \lie g_{\epsilon'}]\subseteq \lie g_{\overline{\epsilon+\epsilon'}}
\end{equation*}
where $\overline{\epsilon+\epsilon'}$ is the remainder of the division of $\epsilon+\epsilon'$ by $m$ (henceforth, we will abuse of notation and write $\epsilon+\epsilon'$ instead of $\overline{\epsilon+\epsilon'}$). In particular, $\lie g_0$ is a subalgebra of $\lie g$ and $\lie g_\epsilon$ is a $\lie g_0$-module for $0\le \epsilon<m$. It is well-known that $\lie g_0$ is simple and that $\lie g_\epsilon$ is a simple $\lie g_0$-module for all $0 \leq \epsilon <  m$ (cf. \cite[Chap. VIII]{Kac}). Moreover, $\lie g_1\cong_{\lie g_0}\lie g_2$ if $m=3$.

If $\lie a$ is a subalgebra of $\lie g$ and $0\le \epsilon<m$, set $\lie a_\epsilon=\lie a\cap\lie g_\epsilon$. Then, $\lie h_0$ is a Cartan subalgebra of $\lie g_0$ and $\lie g_0= \lie n^+_0 \oplus \lie h_0 \oplus \lie n^-_0$. Let $I_0$ be the set of vertices of the Dynkin diagram of $\lie g_0$ and $R_0$ the root system determined by $\lie h_0$.  The simple roots and fundamental weights of $\lie g_0$ determined by the aforementioned triangular decomposition of $\lie g_0$ will be also denoted by $\alpha_i$ and $\omega_i$ for $i\in I_0$, as this should not cause confusion with those of $\lie g$. The root and weight lattices will be denoted by $Q_0$ and $P_0$ and the sets $Q_0^+$ and $P_0^+$ are defined in the usual way. Similarly, $R_0^+$ will denote the corresponding set of positive roots. We let $R_s$ and $R_l$ be the subsets of $R_0^+$ corresponding to short and long roots, respectively.

If $V$ is a finite-dimensional representation of $\lie g$ or of $\lie g_0$, let $\wt(V)$ be the set of weights of $V$.  It is well-known that $(\lie g_\epsilon)_0=\lie h_\epsilon$ and, if $\mu\ne 0$, then $(\lie g_\epsilon)_\mu$ is one-dimensional for all $\epsilon$. We record the following table with information about $\lie g_\epsilon$ according to $\lie g$ and $\sigma$.

$$\begin{tabular}[ht]{|l|c|l|c|c|}
  \hline
    $\quad\mathfrak{g}$ & $m$ & $\lie g_0$ & $\wt(\lie g_1)\setminus\{0\}$ & Dynkin diagram of $\lie g_0$\\
  \hline\hline
   $A_{2}$ & $2$ & $A_1$ & $\pm R_0\cup\{\pm 2\alpha : \alpha\in R_s\}$&  $\underset{1}{\circ}$ \\
   $A_{2n}, n\geq 2$ &  $2$ & $B_n$ & $\pm R_0\cup\{\pm 2\alpha : \alpha\in R_s\}$&  $\underset{1}{\circ} - \underset{2}{\circ} - \dotsb - \underset{n-1}{\circ} \Rightarrow \underset{n}{\circ}$ \\
   $A_{2n-1}, n \geq 2$ &  $2$ &  $C_n$ & $\pm R_s$ &  $\underset{1}{\circ} - \underset{2}{\circ} - \dotsb - \underset{n-1}{\circ} \Leftarrow \underset{n}{\circ}$ \\
   $D_{n+1}, n\geq 3$ &  $2$ & $B_n$ & $\pm R_s$ &  $\underset{1}{\circ} - \underset{2}{\circ} - \dotsb - \underset{n-1}{\circ} \Rightarrow \underset{n}{\circ}$\\
   $E_6 $ & $2$ & $F_4$ & $\pm  R_s$ &   $\underset{1}{\circ} - \underset{2}{\circ} \Leftarrow \underset{3}{\circ} - \underset{4}{\circ}$\\
   $D_{4}$ & $3$ & $G_2$ & $\pm  R_s$ &$ \underset{1}{\circ} \Rrightarrow \underset{2}{\circ}$\\

  \hline
\end{tabular}$$

\subsection{Chevalley basis} \label{s:chevalley}
Fix a Chevalley basis $\cal C=\{x_\alpha^\pm, h_i:\alpha\in R^+,i\in I\}$ for $\lie g$ and, given $\alpha\in R^+$, set $h_\alpha=[x_\alpha^+,x_\alpha^-]$. For the remainder of this subsection, assume  that $\sigma$ is nontrivial.
We review a construction of a basis of $(\lie g_\epsilon)_{\mu}$, where $\mu\in\wt(\lie g_\epsilon)$, in terms of this fixed Chevalley basis. It is well-known that if $\mu\ne 0$, then $\mu=\alpha|_{\lie h_0}$ for some $\alpha\in R$ and $\alpha|_{\lie h_0}=\beta|_{\lie h_0}$ if and only if $\beta=\sigma^j(\alpha)$ for some $0\le j<m$.
Suppose first that $\lie g$ is not of type $A_{2n}$. Then, given $\alpha\in R^+$ and $0\le\epsilon<m$, set
\begin{equation*}
x_{\alpha,\epsilon}^\pm =
\begin{cases}
\sum\limits_{j= 0}^{m-1} \zeta^{j\epsilon}x^\pm_{\sigma^j(\alpha)}, & \quad\text{if } \sigma(\alpha)\ne\alpha,\\
\delta_{0,\epsilon}x_\alpha^\pm,& \quad \text{otherwise,}
\end{cases}
\qquad\text{and}\qquad
\hbar_{\alpha,\epsilon} =
\begin{cases}
\sum\limits_{j= 0}^{m-1} \zeta^{j\epsilon}h_{\sigma^j(\alpha)}, & \quad\text{if } \sigma(\alpha)\ne\alpha,\\
\delta_{0,\epsilon}h_\alpha,& \quad \text{otherwise.}
\end{cases}
\end{equation*}
Notice that
\begin{equation}\label{e:basisrel}
x_{\sigma(\alpha),\epsilon}^\pm = \zeta^{-\epsilon}x_{\alpha,\epsilon}^\pm, \qquad
\hbar_{\sigma(\alpha),\epsilon} = \zeta^{-\epsilon}\hbar_{\alpha,\epsilon}, \qquad x_\alpha^\pm = \frac{1}{\Gamma_\alpha} \tsum_{\epsilon=0}^{m-1} x_{\alpha,\epsilon}^\pm,  \qquad\text{and}\qquad h_\alpha = \frac{1}{\Gamma_\alpha} \tsum_{\epsilon=0}^{m-1} \hbar_{\alpha,\epsilon}.
\end{equation}

Now assume $\lie g$ is of type $A_{2n}$. In this case, for $\alpha\in R^+$, we have
\begin{equation*}
\alpha=\sigma(\alpha)\quad\Leftrightarrow\quad \alpha|_{\lie h_0}\in 2R_s\quad\Leftrightarrow\quad \alpha=\beta+\sigma(\beta) \quad\text{for some}\quad \beta\in R^+.
\end{equation*}
Moreover, if $\beta$ is as above, then $\beta|_{\lie h_0}\in R_s$.
It is easily seen that, for such $\alpha$, we have $x_\alpha^\pm\in\lie g_1$ and $h_\alpha\in\lie g_0$. Thus, we set
\begin{equation*}
x_{\alpha,\epsilon}^\pm= \delta_{1,\epsilon}x_\alpha^\pm \quad\text{and}\quad \hbar_{\alpha,\epsilon}=\delta_{0,\epsilon}h_\alpha.
\end{equation*}
Otherwise, if $\alpha\ne\sigma(\alpha)$, set
\begin{equation*}
x_{\alpha,\epsilon}^\pm =
\begin{cases}
x_\alpha^\pm+(-1)^\epsilon x_{\sigma(\alpha)}^\pm, & \quad\text{if } \alpha|_{\lie h_0}\in R_l,\\
\sqrt 2\left( x_\alpha^\pm+ (-1)^\epsilon x_{\sigma(\alpha)}^\pm\right), & \quad\text{if } \alpha|_{\lie h_0}\in R_s,\\
\end{cases}
\quad\text{and}\quad
\hbar_{\alpha,\epsilon} =
\begin{cases}
h_\alpha+(-1)^\epsilon h_{\sigma(\alpha)}, & \quad\text{if } \alpha|_{\lie h_0}\in R_l,\\
2 \left(h_\alpha+(-1)^\epsilon  h_{\sigma(\alpha)}\right), & \quad\text{if } \alpha|_{\lie h_0}\in R_s.\\
\end{cases}
\end{equation*}
Observe that relations \eqref{e:basisrel} remain valid in this case and that, if $\alpha+\sigma(\alpha)\in R^+$, we have
\begin{equation}\label{e:x2Rs}
x_{\alpha+\sigma(\alpha),1}^\pm =  \frac{s}{4}[x_{\alpha,0}^\pm,x_{\alpha,1}^\pm]=s[ x_\alpha^\pm,x_{\sigma(\alpha)}^\pm] \quad\text{for some}\quad s\in\{\pm 1\}.
\end{equation}

It is now straightforward to check that
\begin{equation*}
(\lie g_{\epsilon})_{\pm\mu} = \mathbb Cx_{\alpha,\epsilon}^\pm \quad\text{if}\quad \alpha|_{\lie h_0}=\mu\in\wt(\lie g_\epsilon)
\end{equation*}
and that $\lie h_\epsilon$ is spanned by the elements $h_{\alpha_i,\epsilon}$ with $i\in I$. In order to fix a basis for these spaces, we proceed as follows. Let $O$ be a complete set of representatives of the orbits of the $\sigma$-action on $R^+$. Then, if $\mu\in\wt(\lie g_\epsilon)\cap Q_0^+\setminus\{0\}$, set
\begin{equation}\label{e:xmuxalpha}
x_{\mu,\epsilon}^\pm = x_{\alpha,\epsilon}^\pm \quad\text{where}\quad \alpha\in O, \  \alpha|_{\lie h_0}=\mu.
\end{equation}
Notice also that there exists a unique injective map $o:I_0\to I$ such that $\alpha_{o(i)}\in O$ for all $i\in I_0$ and set
\begin{equation*}
h_{i,\epsilon} = \begin{cases} \frac{1}{2} \hbar_{o(i),\epsilon},  & \mbox{ if } \lie g \text{ is of type } A_{2n} \text{ and } \alpha_i \in R_s, \\
\hbar_{o(i),\epsilon}, &\text{ otherwise.} \end{cases}
\end{equation*}
Then,
$$\cal C^\sigma(O) = \{x_{\mu,\epsilon}^\pm, h_{i,\epsilon}: 0\le \epsilon<m, \mu\in\wt(\lie g_\epsilon)\cap Q_0^+\setminus\{0\}, i\in I_0\}\setminus\{0\}$$
is a basis of $\lie g$. For notational convenience, for $\lie g$ of type $A_{2n}$, we shall assume that $O$ is chosen in such a way that $s=1$ in \eqref{e:x2Rs}. In the case that $\lie g$ is of type $D_4$ and $m=3$, it is convenient to choose $O$ in a more specific way. Namely, let $j\in I$ be the unique vertex fixed by $\sigma$ and choose $i\in I$ such that $\sigma(i)\ne i$. Then set
\begin{equation*}
O_i = \{\alpha\in R^+:\sigma(\alpha)=\alpha\}\cup\{\alpha_i, \alpha_j+\alpha_i, \alpha_j+\sigma(\alpha_i)+\sigma^2(\alpha_i)\}.
\end{equation*}
The reason for such choices is Theorem \ref{forms} below.
In particular, the brackets of elements of $\cal C^\sigma(O)$ are in the $\mathbb Z$-span of $\cal C^\sigma(O)$. We shall need some of these brackets explicitly for later use and we record them in the next lemma. For notational convenience, we set $x_{\mu,\epsilon}^\pm=0$ if $\mu\notin\wt(\lie g_\epsilon)$ and, if $\alpha\in O$ is such that $\mu=\alpha|_{\lie h_0}$, we set $h_{\mu,\epsilon}=h_{\alpha,\epsilon}$.

\begin{lem}\label{[]tw} Let $0\le \epsilon,\epsilon' <m$. Given $\mu \in R_0^+$, $\nu \in \wt(\lie g_\epsilon)\cap Q_0^+\backslash{\{0\}}$ and  $\eta \in \wt(\lie g_\epsilon)\cap \wt(\lie g_{\epsilon'})\cap Q_0^+\backslash{\{0\}}$, we have:
\begin{enumerate}[(a)]
\item\label{[]tw-weight}  $[h_{\mu,0}, x^\pm_{\nu,\epsilon}]= \pm \nu(h_{\mu,0}) x^\pm_{\nu,\epsilon}$.

\item $[x_{\eta,\epsilon}^+,x_{\eta,\epsilon'}^-]=  \begin{cases} 2h_{\eta,\epsilon+\epsilon'}, & \text{ if } \eta \in R_s \text{ and } \lie g \text{ is of type } A_{2n},  \\ \delta_{\epsilon\epsilon',1}h_{\frac{\eta}{2},0}, & \text{ if } \eta \in 2R_s \text{ and }  \lie g \text{ is of type }  A_{2n}, \\ h_{\eta, \epsilon+\epsilon'}, & \text{ otherwise. } \end{cases}$

\item If $h_{\nu,1}\ne0$, then $[h_{\nu,1},x_{\nu,\epsilon}^\pm]=  \begin{cases} \pm 3 x^\pm_{\nu,\epsilon+1}, &\text{ if } \nu \in R_s \text{ and } \lie g \text{ is of type } A_{2n}, \\ \pm 2 x^\pm_{\nu,\epsilon+1}, & \text{ otherwise. } \end{cases} $\hfill\qedsymbol
\end{enumerate}
\end{lem}

\begin{rem}\label{r:noCbforA2n}
Observe that, unless $\lie g$ is of type $A_{2n}$, the set $\cal C^\sigma_0(O) = \{x_{\mu,0}^\pm, h_{i,0}: \mu\in R_0^+, i\in I_0\}$ is a Chevalley basis of $\lie g_0$ (this does not depend on the more specific choice of $O$ in the case $m=3$). For $\lie g$ of type $A_{2n}$, a Chevalley basis is obtained from $\cal C^\sigma_0(O)$ by replacing the element $h_{j,0}$, where $j$ is the unique element of $I_0$ such that $\alpha_j\in R_s$, by the element $\hbar_{o(j),\epsilon} = 2h_{j,0}$.
\end{rem}

\subsection{Loop algebras}

Given a vector space $\lie a$ over $\mathbb C$, consider its loop space $\tlie a=\lie a\otimes \mathbb{C}[t,t^{-1}]$.
If $\lie a$ is a Lie algebra, then $\tlie a$ is also a Lie algebra, called the loop algebra of $\lie a$, where the bracket is given by
$[x\otimes f, y\otimes g] =[x,y]\otimes fg$ for $x,y\in\lie a$ and $f,g\in\mathbb{C}[t,t^{-1}]$.
Notice that $\lie a\otimes 1$ is a subalgebra of $\tlie a$ isomorphic to $\lie a$ and, if $\lie b$ is a subalgebra of $\lie a$,
then $\tlie b=\lie b\otimes \mathbb{C}[t,t^{-1}]$ is naturally a subalgebra of $\tlie a$.
In particular, $\tlie g = \tlie n^- \oplus \tlie h \oplus \tlie n^+$ and $\tlie h$ is an abelian subalgebra of $\tlie g$.
The elements $x_\alpha^\pm\otimes t^r$ and $h_i\otimes t^r$ with $\alpha\in R^+,i\in I, r\in\mathbb Z$,
form a basis of $\tlie g$ which we shall refer to as a Chevalley basis following \cite{mitz}.

The automorphism $\sigma$ of $\lie g$ can be extended to an automorphism $\widetilde{\sigma}$ of $\tlie g$ by defining
$$\widetilde{\sigma} (x\otimes t^j)= \zeta^j \sigma(x) \otimes t^j,$$
for $x \in \lie g$, $j \in \mathbb{Z}$. Notice that $\widetilde\sigma$ is also an automorphism of order $m$. Let $\tlie g^\sigma$ be the subalgebra of fixed points of $\widetilde\sigma$, which can be described as
$$\tlie g^\sigma = \opl_{\epsilon=0}^{m-1} \tlie g_\epsilon^\sigma \qquad\text{where}\qquad  \tlie g_\epsilon^\sigma = \lie g_\epsilon\otimes t^{m-\epsilon}\mathbb{C}[t^{m}, t^{-m}].$$
The algebra $\tlie g^\sigma$ is called the twisted loop algebra of $\lie g$ (associated to $\sigma$). Define the $\sigma$-invariant subspaces $(\tlie n^\pm_\epsilon)^\sigma$ and $\tlie h^\sigma_\epsilon$ in the obvious manner. In particular, $\tlie g^\sigma_\epsilon = (\tlie n^-_\epsilon)^\sigma \oplus \tlie h^\sigma_\epsilon \oplus (\tlie n^+_\epsilon)^\sigma$.

Given a basis of $\lie g$ of the form $\cal C^\sigma(O)$ as in Subsection \ref{s:chevalley}, the associated elements of the form $x_{\mu,\epsilon}^\pm\otimes t^r, h_{i,\epsilon}\otimes t^r$ with $x_{\mu,\epsilon}, h_{i,\epsilon}\in\cal C^\sigma(O), r\in\mathbb Z, \epsilon \equiv_m -r$, form a basis of $\tlie g^\sigma$ which we also refer to as a Chevalley basis following \cite{mitz}. The next lemma is easily established (a proof can be found in \cite{bia:phd}; see also \cite{CFS}).

\begin{lem}\label{isos} Let $\alpha\in R_0^+$.
\begin{enumerate}[(a)]
\item Suppose $\lie g$ is not of type $A_{2n}$. If $\alpha\in R_l$, then the subspace spanned by $\{x^\pm_{\alpha,0}\otimes t^{mk}, h_{\alpha,0}\otimes t^{mk}  :  k\in\mathbb Z\}$ is a subalgebra of $\tlie g^\sigma$ isomorphic to $\tlie{sl}_2$. Otherwise, if $\alpha\in R_s$, the subspace spanned by $\{x^\pm_{\alpha,\epsilon}\otimes t^{mk-\epsilon}, h_{\alpha,\epsilon}\otimes t^{mk-\epsilon}  :  k\in\mathbb Z,  \ 0 \leq \epsilon < m \}$ is a subalgebra of $\tlie g^\sigma$ isomorphic to $\tlie{sl}_2$.

\item Suppose $\lie g$ is of type $A_{2n}$. If $\alpha\in R_l$, then the subspace spanned by $\{x^\pm_{\alpha,\epsilon}\otimes t^{mk-\epsilon}, h_{\alpha,\epsilon}\otimes t^{mk-\epsilon}  :  k\in\mathbb Z,  \ 0 \leq \epsilon <m \}$ is a subalgebra of $\tlie g^\sigma$ isomorphic to $\tlie{sl}_2$.  If $\alpha \in R_s$, then the subspace spanned by $\{x_{\alpha,\epsilon}^\pm\otimes t^{2k+\epsilon},x_{\alpha+\sigma(\alpha),1}^\pm\otimes t^{2k+1}, h_{\alpha,\epsilon}\otimes t^{2k+\epsilon}  :  k\in\mathbb Z, \ 0 \leq \epsilon \leq m-1 \}$  is a subalgebra of $\tlie g^\sigma$ isomorphic to $\tlie{sl}_3^{\tau}$ where $\tau$ is the nontrivial diagram automorphism of $\lie{sl}_3$.\hfill\qedsymbol
\end{enumerate}
\end{lem}

\subsection{Hyperalgebras}

For an associative algebra $A$ over a field of characteristic zero, $a\in A$, and $k\in\mathbb Z_+$, we set $a^{(k)}=\frac{a^k}{k!}$ and $\tbinom{a}{k} = \frac{a(a-1)\cdots(a-k+1)}{k!}$.
We shall need the following two identities. Suppose that $x_1,\dots,x_r$ are commuting elements of  $A$. Then,
\begin{equation}\label{e:divpsum}
\left(\tsum_{j= 1}^{r} x_j\right)^{(n)} = \tsum_{\substack{(n_1,\dots,n_r)\in\mathbb Z_{\ge 0}^m\\ n_1+\cdots + n_r=n}}^{} (x_1)^{(n_1)}\cdots (x_r)^{(n_r)}
\end{equation}
and
\begin{equation}\label{e:expbin}
\binom{x_1+x_2}{m} = \tsum_{j=1}^m \binom{x_1}{j}\binom{x_2}{m-j}.
\end{equation}

Let $U(\lie a)$ denote the universal enveloping algebra of a Lie algebra $\lie a$.
The PBW theorem implies that the multiplication establishes isomorphisms
$$U(\tlie g)\cong U(\tlie n^-)\otimes U(\tlie h)\otimes U(\tlie n^+) \quad\text{and}\quad U(\tlie g^\sigma)\cong U((\tlie n^-)^\sigma)\otimes U(\tlie h^\sigma)\otimes U((\tlie n^+)^\sigma).$$

Given $\alpha\in R^+$, consider the following power series with coefficients in $U(\tlie h)$:
$$\Lambda_{\alpha}^{\pm}(u) = \tsum_{r=0}^\infty \Lambda_{\alpha,\pm r}u^r = \exp \left(-\tsum_{s=1}^{\infty}\frac{h_{\alpha}\otimes t^{\pm s}}{s} u^{s}\right).$$
We will also need to make use of  elements $\Lambda_{\alpha,r;k}\in U(\tlie h)$, $i\in I, r,k\in\mathbb Z, r\ne 0, k>0$, defined as follows. Consider the endomorphism $\tau_k$ of $U(\tlie g)$ induced by $t\mapsto t^k$ and set
\begin{equation*}
\Lambda_{\alpha;k}^\pm (u) = \tsum_{r\ge 0}^{} \Lambda_{\alpha,\pm r;k}u^r = \tau_k(\Lambda_{\alpha}^\pm(u)) = \exp\left(-\tsum_{s>0}^{} \frac{h_{i}\otimes t^{\pm sk}}{s}u^s\right).
\end{equation*}

We now define the twisted analogues of the above elements. If either $\lie g$ is not of type $A_{2n}$ and $\mu \in R_s$, or $\lie g$ is of type $A_{2n}$ and $\mu\in R_0^+$, set
$$\Lambda_{\mu}^{\sigma,\pm}(u)  = \tsum_{r=0}^\infty \Lambda_{\mu,\pm r}^\sigma u^r =\text{exp} \left( - \tsum_{k=1}^{\infty} \tsum_{\epsilon=0}^{m-1}\frac{h_{\mu,\pm\epsilon} \otimes t^{\pm(mk-\epsilon)}}{mk-\epsilon} u^{mk-\epsilon}\right),$$
and, if $\lie g$ is not of type $A_{2n}$ and $\mu \in R_l$, we set
$$\Lambda_{\mu}^{\sigma,\pm}(u)  = \tsum_{r=0}^\infty \Lambda_{\mu,\pm r}^\sigma u^r = \text{exp} \left( - \tsum_{k=1}^{\infty} \frac{h_{\mu,0} \otimes t^{\pm mk}}{k} u^{k}\right). $$
We may write $\Lambda_{i,r}$ in place of $\Lambda_{\alpha_i,r}$ and similarly for $\Lambda^\sigma_{\alpha_i,r}$.
One can easily check the following relation among twisted and non-twisted versions of the above elements. Given $\mu\in R_0^+$, let $\alpha\in O$ such that $\alpha|_{\lie h_0}=\mu$. Then:
\begin{equation}\label{e:tLvsL}
\Lambda_\mu^{\sigma,\pm}(u)=
\begin{cases} \prod_{j=0}^{m -1} \Lambda_{\sigma^j(\alpha)}^\pm(\zeta^{m-j}u), & \text{if}\quad  \Gamma_\alpha =m,  \\
\Lambda_{\alpha;m}^\pm(u),   & \text{if}\quad  \Gamma_\alpha = 1.  \end{cases}
\end{equation}

Given an order on the Chevalley basis of $\tlie g^\sigma$ and a PBW monomial with respect to this order, we construct an ordered monomial in the elements of the set
$$\tilde{\cal M}^\sigma=\left\{(x^\pm_{\mu,-r}\otimes t^r)^{(k)}, \ \Lambda_{i,k}^{\sigma,\pm}, \ \tbinom{h_{i,0}}{k}  :   \mu\in\wt(\lie g_{-r})\cap Q_0^+\setminus\{0\}, i\in I_0, r\in \mathbb Z, k\in \mathbb Z_+\right\}\setminus\{0\},$$
using the correspondence $(x^\pm_{\mu,-r}\otimes t^r)^k \leftrightarrow (x^\pm_{\mu,-r}\otimes t^r)^{(k)}$, $h_{i,0}^k \leftrightarrow \tbinom{h_{i,0}}{k}$ and $(h_{i,-r}\otimes t^r)^k \leftrightarrow (\Lambda_{i,r}^\sigma)^k$ for $r\ne 0$.
Using the obvious similar correspondence we consider monomials in $U(\lie g)$ formed by elements of
$$\cal M=\left\{ (x_{\alpha}^\pm)^{(k)},\tbinom{h_i}{k}  :  \alpha\in R^+, i\in I, k\in\mathbb Z_+\right\}$$
and in $U(\tlie g)$ formed by elements of
$$\tilde{\cal M}=\left\{(x^\pm_{\alpha}\otimes t^r)^{(k)}, \Lambda_{i,k}^{\pm},\tbinom{h_{i}}{k}  :  \alpha \in R^+, i\in I, k\in \mathbb Z_+, r\in \mathbb Z\right\}.$$
Notice that $\cal M$ can be naturally regarded as a subset of $\tilde{\cal M}$. The set of ordered monomials thus obtained are clearly basis of $U(\tlie g^\sigma)$, $U(\lie g)$, and $U(\tlie g)$, respectively.

Let  $U_{\mathbb Z}(\lie g) \subseteq U(\lie g)$, $U_{\mathbb Z}(\tlie g) \subseteq U(\tlie g)$, and $U_{\mathbb Z}(\tlie g^\sigma)\subseteq U(\tlie g^\sigma)$ be the $\mathbb Z$--subalgebras generated respectively by $\{(x_{\alpha}^\pm)^{(k)}  :  \alpha\in R^+, k\in\mathbb Z_+\}$, $\{(x^\pm_{\alpha}\otimes t^r)^{(k)} :  \alpha \in R^+,r\in\mathbb Z, k\in\mathbb Z_+\}$, and $\{(x^\pm_{\mu,-r}\otimes t^r)^{(k)} : \mu\in\wt(\lie g_{-r})\cap Q_0^+\setminus\{0\}, r\in\mathbb Z, k\in\mathbb Z_+\}$. The following crucial theorem was proved in \cite{kosagz}, in the $U(\lie g)$ case, and in \cite{G,mitz} for the $U(\tlie g)$ and $U(\tlie g^\sigma)$ cases (see also \cite{shari}).

\begin{thm} \label{forms}
The subalgebras $U_{\mathbb Z}(\tlie g^\sigma)$, $U_{\mathbb Z}(\lie g)$, and $U_{\mathbb Z}(\tlie g)$ are free $\mathbb Z$-modules and the sets of ordered monomials constructed from $\tilde{\cal M}^\sigma$, $\cal M$, and $\tilde{\cal M}$ are $\mathbb Z$-basis of $U_{\mathbb Z}(\tlie g^\sigma)$, $U_{\mathbb Z}(\lie g)$, and $U_{\mathbb Z}(\tlie g)$, respectively.\hfill\qedsymbol
\end{thm}

In particular, it follows that the natural map $\mathbb C\otimes_\mathbb ZU_\mathbb Z(\tlie g^\sigma)\to U(\tlie g^\sigma)$ is an isomorphism, i.e., $U_\mathbb Z(\tlie g^\sigma)$ is a an integral form of $U(\tlie g^\sigma)$, and similarly for $U_\mathbb Z(\tlie g)$ and $U_\mathbb Z(\lie g)$. If $\lie a$ is a subalgebra of $\lie g$ preserved by $\sigma$, set
\begin{equation*}
U_\mathbb Z(\tlie a^\sigma) = U(\tlie a^\sigma)\cap U_\mathbb Z(\tlie g^\sigma)
\end{equation*}
and similarly define $U_\mathbb Z(\lie a)$ and $U_\mathbb Z(\tlie a)$ for any $\lie a$ subalgebra of $\lie g$. Then,
\begin{equation*}
\lie a\in\{ \lie g_0,\lie n_0^\pm,\lie h_\epsilon, \tlie h^\sigma_\epsilon,(\tlie n^\pm)^\sigma, \lie n^\pm,\lie h, \tlie n^\pm,\tlie h\} \quad\Rightarrow\quad \mathbb C\otimes_\mathbb Z U_\mathbb Z(\lie a)\stackrel{_\cong}{\longrightarrow} U(\lie a).
\end{equation*}
In fact, $U_\mathbb Z(\lie a)$ is a free $\mathbb Z$-module spanned by monomials formed by elements of $\tilde{\cal M}^\sigma$ (or of $\tilde{\cal M}$) belonging to $U(\lie a)$.

\begin{rem}\label{r:KinGM}
Notice that $U_\mathbb Z(\lie g)=U(\lie g)\cap U_\mathbb Z(\tlie g)$, i.e., Kostant's integral form of $\lie g$ coincides with its intersection with Garland's integral form of $U(\tlie g)$ which allows us to regard $U_\mathbb Z(\lie g)$ as a $\mathbb Z$-subalgebra of $U_\mathbb Z(\tlie g)$. If $\lie g$ is not of type $A_{2n}$, then $U_\mathbb Z(\lie g_0)$ coincides with Kostant's integral form of $\lie g_0$ as well. However, if $\lie g$ is of type $A_{2n}$ this is not true for the reason described in Remark \ref{r:noCbforA2n}.
\end{rem}

Given a field  $\mathbb F$, define the $\mathbb F$--hyperalgebra of $\lie a$ by
$$U_\mathbb F(\lie a) =  \mathbb F\otimes_{\mathbb Z}U_\mathbb Z(\lie a)$$
where $\lie a$ is any of the subspaces considered above. We will refer to $U_\mathbb F(\tlie g^\sigma)$  as the twisted hyper loop algebra of $\lie g$ associated to $\sigma$ over $\mathbb F$.
Clearly, if the characteristic of $\mathbb F$ is zero, the algebra $U_\mathbb F(\tlie g^\sigma)$ is naturally isomorphic to $U(\tlie g_\mathbb F^\sigma)$ where $\tlie g_\mathbb F^\sigma=  \mathbb F\otimes_\mathbb Z\tlie g_\mathbb Z^\sigma$ and $\tlie g_\mathbb Z^\sigma$ is the $\mathbb Z$-span of the Chevalley basis of $\tlie g^\sigma$, and similarly for all algebras $\lie a$ we have considered. For fields of positive characteristic we just have an algebra homomorphism $U(\lie a_\mathbb F)\to U_\mathbb F(\lie a)$ which is neither injective nor surjective. We will keep denoting by $x$ the image of an element $x\in U_\mathbb Z(\lie a)$ in $U_\mathbb F(\lie a)$. Notice that we have
$$U_\mathbb F(\tlie g)=U_\mathbb F(\tlie n^-)U_\mathbb F(\tlie h)U_\mathbb F(\tlie n^+) \quad  \text{ and } \quad  U_\mathbb F(\tlie g^\sigma)=U_\mathbb F((\tlie n^-)^\sigma)U_\mathbb F(\tlie h^\sigma)U_\mathbb F((\tlie n^+)^\sigma).$$

\begin{rem}
If $\lie g$ is of type $A_{2n}$ and the characteristic of $\mathbb F$ is 2, then $U_\mathbb F(\lie g_0)$ is not isomorphic to what is usually called the hyperalgebra of $\lie g_0$ over $\mathbb F$ (which is constructed using Kostant's integral form of $U(\lie g_0)$). Indeed, if $i\in I_0$ is the unique element such that $\alpha_i\in R_s$, then $[x_{\alpha_i,0}^+,x_{\alpha_i,0}^-]=\hbar_{o(i),\epsilon}=2h_{i,0}=0$ in $U_\mathbb F(\lie g_0)$, but this is not zero in the usual hyperalgebra of $\lie g_0$ over $\mathbb F$. On the other hand, if the characteristic of $\mathbb F$ is not 2, then $U_\mathbb F(\lie g_0)$ is isomorphic to the usual hyperalgebra of $\lie g_0$ over $\mathbb F$ (details can be found in \cite{bia:phd}). For this reason, we shall not work with fields of characteristic 2 when $\lie g$ is of type $A_{2n}$.
\end{rem}

\begin{rem}\label{r:Lambda3i}
It was proved in \cite{G} that the elements $\Lambda_{i,r;k}$ are part of a $\mathbb Z$-basis of the free $\mathbb Z$-module $U_\mathbb Z(\tlie h)$. In particular, their images in $U_\mathbb F(\tlie g)$ are nonzero.
\end{rem}

\subsection{Hopf algebra structure}

We denote by $\Delta, S, \epsilon$ the usual comultiplication, antipode, and counit of $U(\lie a)$ and let $U(\lie g)^0$ denote the augmentation ideal, i.e., the kernel of $\epsilon$. Set also $U_\mathbb Z(\lie a)^0=U_\mathbb Z(\lie a)\cap U(\lie a)^0$.

Notice that the Hopf algebra structure on $U(\tlie g)$ induces a Hopf algebra structure over $\mathbb Z$ on $U_\mathbb Z(\lie a)$ for $\lie a\in\{ \lie g, \lie n^\pm, \lie h,\tlie g,\tlie n^\pm, \tlie h,\tlie g^\sigma,(\tlie n^\pm)^\sigma, \tlie h^\sigma\}$. This in turn induces a Hopf algebra structure on $U_\mathbb F(\lie a)$. Notice also that, if $x\in\lie a$ is such that $x^{(k)}\in U_\mathbb Z(\lie a)$ for all $k>0$, then
\begin{gather}\label{e:comultdp}
\Delta(x^{(k)})= \tsum_{l=0}^{k} x^{(l)}\otimes x^{(l-k)}.
\end{gather}
Also,  for all $k>0$ and $\mu\in R_0^+$, we have
\begin{equation}\label{e:comutLambdaser}
\Delta(\Lambda_{\alpha;k}^\pm(u)) = \Lambda_{\alpha;k}^\pm(u)\otimes \Lambda_{\alpha;k}^\pm(u) \qquad\text{and}\qquad \Delta(\Lambda_\mu^{\sigma,\pm}(u)) = \Lambda_\mu^{\sigma,\pm}(u) \otimes \Lambda_\mu^{\sigma,\pm}(u).
\end{equation}
Notice that \eqref{e:comutLambdaser} is equivalent to
\begin{equation*}
\Delta(\Lambda_{\alpha,\pm r;k}) = \sum_{s=0}^r \Lambda_{\alpha,\pm s;k}\otimes\Lambda_{\alpha,\pm (r-s);k} \qquad\text{and}\qquad  \Delta(\Lambda_{\mu,\pm r}^\sigma) = \sum_{s=0}^r \Lambda_{\mu,\pm s}^\sigma\otimes\Lambda_{\mu,\pm (r-s)}^\sigma
\end{equation*}
for all $r\ge 0$.

\subsection{Some auxiliary identities}

The next two lemmas are crucial in the proof of Theorem \ref{forms} and will also be crucial in the study of finite-dimensional representations of hyper loop algebras. Lemma \ref{basicrel} was originally proved in the $U(\tlie g)$ setting in \cite[Lemma 7.5]{G} (our statement follows that of \cite[Lemma 1.3]{hyperlar}). Given $\alpha\in R^+$ and $s\in \mathbb Z$, consider the following power series with coefficients in $U(\tlie n^-)$:
$$X_{\alpha;s,\pm}(u) = \tsum_{r=1}^\infty  (x^-_{\alpha}\otimes t^{\pm(r+s)}) u^{r}.$$

\begin{lem} \label{basicrel} Let $\alpha \in R^+$, $k,l \in \mathbb N$. Then,
$$(x^+_{\alpha}\otimes t^{\mp s})^{(l)}(x^-_{\alpha}\otimes t^{\pm(s+1)})^{(k)} = (-1)^l \left((X_{\alpha;s,\pm}^-(u))^{(k-l)}\Lambda_{\alpha}^{\pm}(u)\right)_k \mod U(\tlie g)_\mathbb Z U(\tlie n^+)_\mathbb Z^0,$$
where the subindex $k$ means the coefficient of $u^k$ of the above power series.\hfill\qedsymbol
\end{lem}

The next lemma establishes the ``twisted version'' of Lemma \ref{basicrel}. If either $\lie g$ is not of type $A_{2n}$ and $\mu \in R_s$, or $\lie g$ is of type $A_{2n}$ and $\mu \in R_0^+$,  define
$$ X^{\sigma}_{\mu;s,\pm}(u) = \tsum_{k=1}^\infty \tsum_{\epsilon =0}^{m-1} (x^-_{\mu,\mp(s-\epsilon)} \otimes t^{\pm(mk+s-\epsilon)}) u^{mk-\epsilon},$$
for all $s\in \mathbb Z$. For $\lie g$ not of type $A_{2n}$ and $\mu \in R_l$, set
$$ X^{\sigma}_{\mu;ms,\pm}(u) = \tsum_{k=1}^\infty (x^-_{\mu,0} \otimes t^{\pm(m(k+s))}) u^{k},$$
for all $s\in \mathbb Z$.

\begin{lem} \label{basicreltw} Let $l,k,s\in\mathbb Z$ and $0\le l\le k$.
\begin{enumerate}[(a)]
\item If either $\lie g$ is not of type $A_{2n}$ and $\mu \in R_s$, or if $\lie g$ is of type $A_{2n}$ and $\mu \in R_l$, we have
$$(x_{\mu,\pm s}^+\otimes t^{\mp s})^{(l)}_{} (x_{\mu,\mp (s+1)}^-\otimes t^{\pm(s+1)})^{(k)} =(-1)^l \left( ( X^{\sigma}_{\mu;s,\pm}(u))^{(k-l)}(u)   \Lambda_{\mu}^{\sigma,\pm}(u) \right)_{k} \mod U_\mathbb Z(\tlie{g}^\sigma)U_\mathbb Z((\tlie{n}^+)^\sigma)^0.$$

\item If $\lie g$ is not of type $A_{2n}$ and $\mu \in R_l$, we have
$$(x_{\mu,0}^+\otimes t^{\mp ms)})^{(l)} (x_{\mu,0}^-\otimes t^{\pm m(s+1)})^{ (k)} =(-1)^l \left( (X^{\sigma}_{\mu;ms,\pm}(u))^{(k-l)}   \Lambda_{\mu}^{\sigma,\pm}(u) \right)_{k}  \mod U_\mathbb Z(\tlie{g}^\sigma)U_\mathbb Z((\tlie{n}^+)^\sigma)^0.$$

\item If $\lie g$ is of type $A_{2n}$ and $\mu\in R_s$, we have
\begin{enumerit}
    \item $(x_{\mu,0}^+\otimes t^{\pm s})^{ (2k-a)}  (x_{2\mu,1}^-\otimes t^{\mp(2s-1)})^{(k)}  =- \left((X^{\sigma}_{\mu;0,\pm}(u))^{(a)} \Lambda_{\mu}^{\sigma,\pm}(u) \right)_{k} \mod U_\mathbb Z(\tlie{g}^\sigma)U_\mathbb Z((\tlie{n}^+)^\sigma)^0$

    for $a=0,1$.

  \item $\begin{array}{lcl}(x_{\mu,0}^+\otimes 1)^{(2k)}  (x_{2\mu,1}^-\otimes t)^{(k+r)} &=& (x^-_{2\mu,1}\otimes t^{})^{(r)}\Lambda_{\mu,k}^{\sigma}+ \sum_{j=1}^{k-1} X_{j} \Lambda_{\mu,k-j}^{\sigma} \; \mod U_\mathbb Z(\tlie{g}^\sigma)U_\mathbb Z((\tlie{n}^+)^\sigma)^0  \end{array}$

where $r\in \mathbb Z_+$  and $X_{j}$ is a certain $\mathbb Z$-linear combination of elements of the form
$$(x^-_{\mu, 1} \otimes t^{})^{(s_1)}\cdots (x^-_{\mu,k} \otimes t^{k})^{(s_{k})}(x^-_{2\mu,1} \otimes t)^{(s_1')}\cdots (x^-_{2\mu,1} \otimes t^{k})^{(s_{k'})}$$  with  $\sum_m s_m+2\sum_n s_n = 2r$ and $\sum_m ms_m'+\sum_n ns_n = r+j$.

  \item $(x_{2\mu,1}^+\otimes t^{\mp (2s+1))})^{(l)} (x_{2\mu,1}^-\otimes t^{\pm (2s+3)})^{(k)} =(-1)^l \left( (X^{\sigma}_{2\mu;ms,\pm}(u))^{(k-l)}   \Lambda_{2\mu}^{\sigma,\pm}(u) \right)_{k}$

      $\mod U_\mathbb Z(\tlie{g}^\sigma)U_\mathbb Z((\tlie{n}^+)^\sigma)^0$, where $$X^{\sigma}_{2\mu;ms,\pm}(u) = \tsum_{k=1}^{} x_{\mu,1}\otimes t^{\pm(m(k+s)+1)}u^k \quad\text{and}\quad  \Lambda_{2\mu}^{\sigma,\pm}(u) = \text{exp} \left( - \tsum_{k=1}^{\infty} \frac{h_{\mu,0} \otimes t^{\pm mk}}{k} u^{k}\right).$$

  \item $(x_{2\mu,1}^+\otimes t^{1-2s})^{(rk)}  (x_{\mu,1}^-\otimes t^{s})^{(2rk+r)} = \sum_{j=1}^{k+1} (x^-_{\mu,s+j-1}\otimes t^{s+j-1})^{(r)} \Lambda^\sigma_{rk+r-rj}+  X $ $\mod U_\mathbb Z(\tlie{g}^\sigma)U_\mathbb Z((\tlie{n}^+)^\sigma)^0$, where $X$ is a sum of elements belonging to
$(\tprod_{i}^{}(x^-_{2\mu, -a_{i}} \otimes t^{a_i})^{(r_i)})(\tprod_{j}^{}(x^-_{\mu,-b_j} \otimes t^{b_j})^{(s_j)})U_\mathbb F(\tlie h^\sigma)$ with $a_i,b_j \in \mathbb Z$ and $r_i,s_j\in \mathbb Z_+$ satisfying $0\le r_i,s_j < k$.
\end{enumerit}
\end{enumerate}

\end{lem}

\proof Items (a) and (b) are direct consequences of Lemmas \ref{basicrel} and  \ref{isos}. All parts of (c) are particular consequences of \cite[Lemma 4.4.1(ii)]{mitz}.
\endproof

A proof of the following identity can be found in \cite[Lemma 26.2]{H}.
\begin{equation}\label{e:comutx+x-}
(x_{\alpha}^+)^{(l)}(x_{\alpha}^-)^{(k)} = \tsum_{m=0}^{\rm{min}\{k,l\}} (x_{\alpha}^-)^{(k-m)} \tbinom{h_{\alpha}-k-l+2m}{m}  (x_{\alpha}^+)^{(l-m)} \qquad\text{for all}\qquad k,l\ge 0, \alpha\in R^+.
\end{equation}
Lemmas \ref{basicrel} and \ref{basicreltw} can be regarded as ``loop'' versions of \eqref{e:comutx+x-}. It immediately follows that
\begin{equation*}
(x_{\alpha,0}^+)^{(l)}(x_{\alpha,0}^-)^{(k)} = \tsum_{m=0}^{\rm{min}\{k,l\}} (x_{\alpha,0}^-)^{(k-m)} \tbinom{\hbar_{\alpha,0}-k-l+2m}{m}  (x_{\alpha,0}^+)^{(l-m)} \qquad\text{for all}\qquad k,l\ge 0, \alpha\in R_0^+.
\end{equation*}
The next identity is easily deduced from Lemma \ref{[]tw}\eqref{[]tw-weight}.
\begin{equation}\label{comutxhtw}
\tbinom{h_{i,0}}{l}(x^\pm_{\mu,- r} \otimes t^r)^{(k)} = (x^\pm_{\mu,- r} \otimes t^r)^{(k)}\tbinom{h_{i,0}\pm k\mu(h_{i,0})}{l} \quad\text{for all}\quad k,l>0, r\in\mathbb Z, i\in I_0, x^\pm_{\mu,- r}\in \cal C^\sigma(O).
\end{equation}

Given $x^\pm_{\mu,- r}\in \cal C^\sigma(O)$ and $k\ge 0$, define the degree of $(x^\pm_{\mu,- r} \otimes t^r)^{(k)}$ to be $k$. For a monomial of the form $(x^\pm_{\mu_1,- r_1} \otimes t^{r_1})^{(k_1)} \cdots (x^\pm_{\mu_l,- r_l} \otimes t^{r_l})^{(k_l)}$ (choice of $\pm$ fixed) define its degree to be $k_1+\cdots +k_l$. The following result was proved in \cite[Lemma 4.2.13]{mitz}.

\begin{lem}\label{sbrw}
Let $r,s\in\mathbb Z, \mu\in\wt(\lie g_{-r}), \nu\in\wt(\lie g_{-s})$, and $k,l\in\mathbb Z_+$.  Then $(x^\pm_{\mu,- r} \otimes t^r)^{(k)}(x^\pm_{\nu,-s} \otimes t^s)^{(l)}$ is in the $\mathbb Z$--span of $(x^\pm_{\nu,-s} \otimes t^s)^{(l)}(x^\pm_{\mu,- r} \otimes t^r)^{(k)}$ together with monomials of degree strictly smaller than $k+l$. \hfill\qedsymbol
\end{lem}

The next identity is trivially established.
\begin{equation}\label{form1}
(x^\pm_{\mu,- r} \otimes t^r)^{(k)}(x^\pm_{\mu,- r} \otimes t^r)^{(l)} = \tbinom{k+l}{k}(x^\pm_{\mu,-r} \otimes t^r)^{(k+l)} \qquad\text{for all}\qquad r\in\mathbb Z, x^\pm_{\mu,- r}\in \cal C^\sigma(O).
\end{equation}
Notice that, if $\mathbb F$ has characteristic $p>0$, \eqref{form1} implies that $((x^\pm_{\mu,- r} \otimes t^r)^{(k)})^p=0$ in $U_\mathbb F(\tlie g^\sigma)$.

In the case of $\lie g$ is of type $A_{2n}$, we will also need the following identity on Heisenberg algebras whose checking is straightforward.
Let $\lie H$ be the $3$-dimensional Heisenberg algebra $\lie H$ generated by elements $x, y, z$ such that $[x,y]=z$ and $z$ being central. Then, the following identity holds in $U(\lie H)$:
\begin{equation}\label{e:heis}
(x+y)^{(n)}=\tsum_{\substack{0\le k \le n \\ n\equiv_2 k}}^{} (-z/2)^{(\frac{n-k}{2})}  \tsum_{r=0}^{k}  x^{(r)}y^{(k-r)}.
\end{equation}

\subsection{Discrete valuation rings and evaluation maps} \label{s:evmaps}

Following \cite{hyperlar}, given an algebraically closed field $\mathbb F$, let $\mathbb A$ be a Henselian discrete valuation ring of characteristic zero having $\mathbb F$ as its residue field (cf. \cite[Theorem 12.4.1, Example 18.3.4(3)]{Ido}). Set $U_\mathbb A(\lie a)=\mathbb A\otimes_\mathbb Z U_\mathbb Z(\lie a)$ whenever $U_\mathbb Z(\lie a)$ has been defined. Clearly
\begin{equation*}
U_\mathbb F(\lie a)\cong \mathbb F\otimes_\mathbb A U_\mathbb A(\lie a).
\end{equation*}

\begin{lem} \label{rootsinA} Let $\mathbb A$ and $\mathbb F$ be as above.
\begin{enumerate}[(a)]
\item If either $m\ne 3$ or the characteristic of $\mathbb F$ is not $3$, then $\zeta\in\mathbb A$.
\item If $m=3$ and the characteristic of $\mathbb F$ is $3$, then $\mathbb A[\zeta]$ is also a discrete valuation ring with the same residue field $\mathbb F$.
\item If the characteristic of $\mathbb F$ is not $2$, then $\sqrt2\in\mathbb A$.
\end{enumerate}
\end{lem}

\proof If $m\ne 3$ then $\zeta=-1$ and part (a) follows immediately. If the characteristic of $\mathbb F$ is not $3$ and $m=3$, then the $3$-cyclotomic polynomial $\Phi_3(x)=x^2+x+1 \in \mathbb A[x]$ has a root in $\mathbb A$ since its reduction to $\mathbb F$ splits ($\mathbb F$ is algebraically closed) and the conclusion of part (a) follows from Hensel's property of $\mathbb A$ (cf. \cite[Theorem 18.1.2]{Ido}). Similarly, if the characteristic of $\mathbb F$ is not $2$, then $\phi(x)=x^2-2 \in \mathbb A[x]$ splits in $\mathbb F$ and we conclude that $\sqrt2\in\mathbb A$ by using Hensel's property once more. Finally, for proving (b), note that $\phi(x):=\Phi_3(x+1)=x^2+3x+3$ is an Eisenstein polynomial relative to the maximal ideal of $\mathbb A$ (since $3$ belongs to this ideal). Hence, by considering the isomorphism $\rho: \mathbb A[x] \to \mathbb A[x]$ defined by $x\mapsto x+1$ we have $\mathbb A[\zeta]\cong \frac{\mathbb A[x]}{(\Phi_3(x))}\cong \frac{\mathbb A[x]}{(\phi(x))}$ and the claim follows from \cite[Ch. 1, Prop. 17]{serre}.
 \endproof

By the statements of parts (a) and (b) of the Lemma \ref{rootsinA}, we can and do assume that $\zeta\in\mathbb A$.

\begin{rem}
By working with $\mathbb A$ instead of $\mathbb Z$ (and assuming that $\zeta\in\mathbb A$) we make viable several constructions such as the evaluation maps which will be presented next. Also, the extra care in the choice of $O$ for the case $m=3$ which we made in the Subsection \ref{s:chevalley} becomes unnecessary, i.e., Theorem \ref{forms} remains valid replacing $\mathbb Z$ by $\mathbb A$ without the special choice of $O$.
\end{rem}

\begin{prop}
$U_\mathbb A(\tlie g^\sigma)$ is an $\mathbb A$-subalgebra of $U_\mathbb A(\tlie g)$ and the inclusion $U_\mathbb A(\tlie g^\sigma)\hookrightarrow U_\mathbb A(\tlie g)$ induces an inclusion of $\mathbb F$-algebras $U_\mathbb F(\tlie g^\sigma)\hookrightarrow U_\mathbb F(\tlie g)$.
\end{prop}

\proof
Given $k\in\mathbb Z$ and $\mu\in\wt(\lie g_k)$, let $\alpha\in O$ be such that $\alpha|_{\lie h_0}=\mu$ so that $x_{\mu,k}^\pm \otimes t^{-k} = \sum_{j=0}^{\Gamma_\alpha} \zeta^{jk}x^\pm_{\sigma^j(\alpha)}\otimes t^{-k}$.
Assume first that $\alpha+\sigma(\alpha)\notin R$. Then, since $\{ x_{\sigma^j(\alpha)}^\pm : j=0,\dots,m-1\}$ is a commuting set of elements of $\lie g$,
it follows from \eqref{e:divpsum} (using that $\zeta\in\mathbb A$) that  $\left(x^\pm_{\mu,k}\otimes t^{-k}\right)^{(n)}\in U_\mathbb A(\tlie g)$.

If $\alpha+\sigma(\alpha)\in R$, then $\lie g$ is of type $A_{2n}$ and, by our assumptions, $\mathbb F$ has characteristic different than 2. In particular, $2,\sqrt 2\in\mathbb A^\times$. Moreover, $\alpha|_{\lie h_0}\in R_s$ and $x_{\mu,k}^\pm\otimes t^{-k} = \sqrt2\left(x_{\alpha}^\pm\otimes t^{-k}+(-1)^kx_{\sigma(\alpha)}^\pm\otimes t^{-k}\right)\in U_\mathbb A(\tlie g)$. Furthermore, the subalgebra of $\lie n^\pm$ generated by $x_{\alpha}^\pm$ and $x_{\sigma(\alpha)}^\pm$ is a Heisenberg subalgebra with central element $x_{\alpha+\sigma(\alpha)}^\pm$. Then, it follows from \eqref{e:heis} that
$\left(x^\pm_{\mu,k}\otimes t^{-k}\right)^{(n)}$ is in the $\mathbb A$-span of elements of the form
$$(x^\pm_{\alpha}\otimes t^{-k})^{(n_1)}(x^\pm_{\sigma(\alpha)}\otimes t^{-k})^{(n_2)}(x^\pm_{\alpha+\sigma(\alpha)}\otimes t^{-2k})^{(n_3)}  \quad\text{with}\quad n_1+n_2+2n_3 = n.$$
with coefficients in $\mathbb A^\times$. In particular, $\left(x^\pm_{\mu,k}\otimes t^{-k}\right)^{(n)}\in U_\mathbb A(\tlie g)$. This completes the proof of the first statement.

For proving the second statement, it suffices to prove that every $\mathbb A$-basis element of $U_\mathbb A(\tlie g^\sigma)$ is written as an $\mathbb A$-linear combination of basis elements of $U_\mathbb A(\tlie g)$ and at least one of the coordinates lie in $\mathbb A^\times$. We already observed above that this is true for the elements $\left(x^\pm_{\mu,k}\otimes t^{-k}\right)^{(n)}$. From here it is easy to deduce that the same property holds for the basis elements of $U_\mathbb A((\tlie n^\pm)^\sigma)$ of the form $\prod_j (x^\pm_{\mu_j,k_j}\otimes t^{-k_j})^{(n_j)}$ with $\mu_j \in \wt(\lie g_{k_j})\cap Q_0^+$ and $\mu_j \ne \mu_i$ for all $i\ne j$. In fact, as in the previous paragraph, each factor of this product can be written as a linear combination of some elements in $U_\mathbb A(\tlie g)$ with coefficients in $\mathbb A^\times$. By taking in each of these linear combinations a term with the highest exponent (for instance, in a sum
$$\sum_i a_i(x^\pm_{\beta_i}\otimes t^{-k_i})^{(n_i)}(x^\pm_{\sigma(\beta_i)}\otimes t^{-k_i})^{(n-n_i)} \quad\text{where }\quad a_i\in \mathbb A^\times \quad \text{ and } \quad 0\le n_i \le n,$$
take $(x^\pm_{\beta_i}\otimes t^{-k_i})^{(n_i)}$ with $n_i=n$, or $(x^\pm_{\sigma(\beta_i)}\otimes t^{-k_i})^{(n_i)}$ with $n_i=n$), we conclude that the product of these collected terms has coefficient in $\mathbb A^\times$ and it appears only once in the expansion of the product $\prod_j (x^\pm_{\mu_j,k_j}\otimes t^{-k_j})^{(n_j)}$. If this term is in the PBW-order of the Chevalley basis of $\tlie g$, we are done. Otherwise, we use Lemma \ref{sbrw} to produce the term in the correct order and notice that this term has the same coefficient as before.

It remains to consider the basis elements of $U_\mathbb A(\tlie h^\sigma)$. For the elements of the form $\tbinom{h_{o(i)}}{k}$ for $i\in I_0$ and $k\in \mathbb Z_+$, recall that
$$h_{i,\epsilon} = \tsum_{j=0}^{\Gamma_{\alpha_{o(i)}}-1} \zeta^{j\epsilon}h_{\sigma^j(o(i))}.$$
Using \eqref{e:expbin} we see that $\tbinom{h_{i,0}}{k}$ is in the span of elements of the form
\begin{equation*}
\tprod_{j=1}^{\Gamma_{\alpha_{o(i)}}} \tbinom{h_{\sigma^j(o(i))}}{k_j} \qquad\text{with}\qquad \tsum_{j=1}^{\Gamma_{\alpha_{o(i)}}} k_j = k.
\end{equation*}
Furthermore, the coefficient of the term $\tbinom{h_{o(i)}}{k}$ is $1$. Next, notice that, if $o(i)$ is not fixed by $\sigma$, then
\begin{equation*}
\Lambda_{\alpha_i}^{\sigma,\pm}(u)= \tprod_{j=0}^{m-1} \Lambda_{\sigma^j(\alpha_{o(i)})}^\pm(\zeta^{m-j}u).
\end{equation*}
Given $r>0$, it follows that $\Lambda_{\alpha_i,\pm r}^\sigma$ is in the $\mathbb A$-span of elements of the form
\begin{equation*}
\tprod_{j=1}^m \Lambda_{\sigma^j(\alpha_{o(i)}),\pm r_j} \qquad\text{with}\qquad \tsum_{j=1}^m r_j = r.
\end{equation*}
Moreover, the coefficient of the term $\Lambda_{\alpha_{o(i)},\pm r}$ is $\zeta^{\pm mr}\in\mathbb A^\times$. Finally, if $o(i)$ is fixed by $\sigma$, then
\begin{equation*}
\Lambda_{\alpha_i}^{\sigma,\pm}(u)= \Lambda_{\alpha_{o(i)};m}^\pm(u).
\end{equation*}
In particular, $\Lambda_{i,r}^\sigma = \Lambda_{o(i),r;m}$ is an element of an $\mathbb A$-basis of $U_\mathbb A(\tlie h)$ by Remark \ref{r:Lambda3i}. The case of products of such elements is easily deduced from here.
\endproof

Recall from \cite[Proposition 3.3]{hyperlar} that there exists a natural surjective map of $\mathbb A$-algebras $\ev:U_\mathbb A(\tlie g)\to U_\mathbb A(\lie g)\otimes_\mathbb A\mathbb A[t,t^{-1}]$ induced by the identity map $\tlie g\to\tlie g$. Denote by $\ev^\sigma$ the composition of $\ev$ with the inclusion given by the above proposition. This induces an $\mathbb F$-algebra map
\begin{equation*}
\ev^\sigma:U_\mathbb F(\tlie g^\sigma)\to U_\mathbb F(\lie g)\otimes \mathbb F[t,t^{-1}].
\end{equation*}
In particular, given $a\in\mathbb F^\times$, we have an $\mathbb F$-algebra map
\begin{equation*}
\ev^\sigma_a: U_\mathbb F(\tlie g^\sigma)\to U_\mathbb F(\lie g)
\end{equation*}
given by the composition of $\ev^\sigma$ with the evaluation map $U_\mathbb F(\lie g)\otimes \mathbb F[t,t^{-1}]\to U_\mathbb F(\lie g)$ defined by $x\otimes f(t)\mapsto f(a)x$. Similarly, we have a map $\ev_a: U_\mathbb F(\tlie g)\to U_\mathbb F(\lie g)$. We shall refer to the map $\ev$ (respectively, $\ev^\sigma$) as the (twisted) formal evaluation map and to the map $\ev_a$ (respectively, $\ev^\sigma_a$) as the (twisted) evaluation map at $a$.
Notice that
\begin{equation}\label{e:ev(gen)}
\ev_a((x^\pm_{\alpha} \otimes t^r)^{(k)}) = a^{rk}(x_{\alpha}^\pm)^{(k)} \qquad\text{and}\qquad \ev_a( \Lambda_{\alpha,r})=(-a)^{r} \tbinom{h_{\alpha}}{|r|}.
\end{equation}
In particular,
\begin{equation*}
\ev_a^\sigma((x^\pm_{\mu,- r} \otimes t^r)^{(k)}) = a^{rk}(x_{\mu,-r}^\pm)^{(k)} \quad \text{ for all } \quad r \in \mathbb Z.
\end{equation*}
Moreover, if $\mu\in R_0^+$ and $\alpha\in O$ are such that $\alpha|_{\lie h_0}=\mu$, then \eqref{e:tLvsL} implies
\begin{equation}\label{e:evsigma}
\ev_a^\sigma( \Lambda_{\mu,\pm r}^{\sigma})=
\begin{cases} (-a)^{\pm r} \tsum_{\substack{ r_0,\cdots, r_{m-1} \in \mathbb Z_+ \\ r_0+\cdots+r_{m-1}= r}}^{} \tprod_{j=0}^{m-1} \zeta^{-jr_j}\tbinom{h_{\sigma^j(\alpha)}}{r_j},  & \text{ if }  \Gamma_\alpha =m,\\
(-a^m)^{\pm r} \tbinom{h_{\alpha}}{r}, & \text{ if }  \Gamma_\alpha = 1 \end{cases}
\end{equation}
for all $ r \in \mathbb Z_+$.

\section{Finite-dimensional representations of $U_\mathbb F(\lie g)$ and $U_\mathbb F(\tlie g)$}\label{s:gandtg}

In this section we review some results on finite-dimensional representations of $U_\mathbb F(\lie g)$ and of $U_\mathbb F(\tlie g)$ which will be relevant for our purposes.

\subsection{$U_\mathbb F(\lie g)$-modules}

All results stated here can be found in \cite{H} in the characteristic zero setting. The literature for the positive characteristic setting is more often found in the context of algebraic groups as in \cite{janb} and a more detailed review in the present context can be found in \cite[Section 2]{hyperlar}. Evidently, similar results apply for $U_\mathbb F(\lie g_0)$-modules.
Let $\le$ denote the usual partial order $\lie h^*$ and let $\mathcal W$ be the  Weyl group of $\lie g$. The longest element of $\cal W$ is denoted by $w_0$.

Let $V$ be a $U_\mathbb F(\lie g)$--module. A nonzero vector $v\in V$ is called a weight vector if there exists $\mu\in U_\mathbb F(\lie h)^*$ such that $hv=\mu(h)v$ for all $h\in U_\mathbb F(\lie h)$. The subspace spanned by weight vectors of weight $\mu$ (the weight space of weight $\mu$) will be denoted by $V_\mu$. If $V=\opl_{\mu\in U_\mathbb F(\lie h)^*}^{} V_\mu$, then $V$ is said to be a weight module. If $V_\mu\ne 0$, $\mu$ is said to be a weight of $V$ and we let
$$\wt(V) = \{\mu\in U_\mathbb F(\lie h)^*:V_\mu\ne 0\}.$$
Notice that we have an inclusion $P\hookrightarrow U_\mathbb F(\lie h)^*$ determined by
\begin{equation*}
\mu\left(\tbinom{h_i}{k}\right) = \tbinom{\mu(h_i)}{k} \quad\text{and}\quad \mu(xy)=\mu(x)\mu(y) \quad\text{for all}\quad  i\in I,k\ge 0, x,y\in U_\mathbb F(\lie h).
\end{equation*}
In particular, we can consider the partial order $\le$ on $U_\mathbb F(\lie h)^*$ given by $\mu\le\lambda$ if $\lambda-\mu\in Q^+$ and we have
\begin{equation*}
(x_\alpha^\pm)^{(k)} V_\mu \subseteq V_{\mu\pm k\alpha}\quad\text{for all}\quad \alpha\in R^+, k>0,\mu\in U_\mathbb F(\lie h)^*.
\end{equation*}
If $V$ is a weight-module with finite-dimensional weight spaces, its character is the function $\ch(V):U_\mathbb F(\lie h)^*\to \mathbb Z$ given by $\ch(V)(\mu)=\dim V_\mu$. As usual, if $V$ is finite-dimensional, $\ch(V)$ can be regarded as an element of the group ring $\mathbb Z[U_\mathbb F(\lie h)^*]$ where we denote the element of $\mathbb Z[U_\mathbb F(\lie h)^*]$ corresponding to $\mu\in U_\mathbb F(\lie h)^*$ by $e^\mu$. Notice that the group ring $\mathbb Z[P]$ can be naturally regarded as a subring of $\mathbb Z[U_\mathbb F(\lie h)^*]$ and, moreover, the action of $\cal W$ on $P$ extends naturally to an action of $\cal W$ on $\mathbb Z[P]$ by ring homomorphisms.

If  $v$ is a weight vector such that $(x_\alpha^+)^{(k)}v=0$ for all $\alpha\in R^+, k>0$, then $v$ is said to be a highest-weight vector and $V$ is said to be a highest-weight module if it is generated  by a highest-weight vector. Similarly, one defines the notions of lowest-weight vectors and modules by replacing $(x_\alpha^+)^{(k)}$ by $(x_\alpha^-)^{(k)}$.

\begin{thm}\label{t:rh}Let $V$ be a $U_\mathbb F(\lie g)$-module.
\begin{enumerate}[(a)]
\item If $V$ is finite-dimensional, then $V$ is a weight-module. Moreover, $V_\mu\ne 0$ only if $\mu\in P$, and $\dim V_\mu = \dim V_{w\mu}$ for all $w\in\cal W$. In particular, $\ch(V)\in\mathbb Z[P]^\cal W$.
\item If $V$ is a highest-weight module of highest weight $\lambda$, then $\dim(V_{\lambda})=1$ and $V_{\mu}\ne 0$ only if $\mu\le \lambda$. Moreover, $V$  has a unique maximal proper submodule and, hence, also a unique irreducible quotient. In particular, $V$ is indecomposable.
\item For each $\lambda\in P^+$, the $U_\mathbb F(\lie g)$-module $W_\mathbb F(\lambda)$ generated by a vector $v$ satisfying the defining relations
$$(x_\alpha^+)^{(k)}v=0, \quad hv=\lambda(h)v \quad \text{and} \quad (x_i^-)^{(l)}v=0,$$
for all $\alpha\in R^+, h\in U_\mathbb F(\lie h), i\in I, k>0$ and $l>\lambda(h_i)$, is nonzero and finite-dimensional. Moreover, every finite-dimensional highest-weight module of highest weight $\lambda$ is a quotient of $W_\mathbb F(\lambda)$.
\item If $V$ is finite-dimensional and irreducible, then there exists a unique $\lambda\in P^+$ such that $V$ is isomorphic to the irreducible quotient $V_\mathbb F(\lambda)$ of $W_\mathbb F(\lambda)$.
\item The character of $W_\mathbb F(\lambda), \lambda\in P^+$, is given by the Weyl character formula. In particular, $\mu\in\wt(W_\mathbb F(\lambda))$ if, and only if, $w\mu\le\lambda$ for all $w\in\cal W$. Moreover, $W_\mathbb F(\lambda)$ is a lowest-weight module with lowest weight $w_0\lambda$.\hfill\qedsymbol
\end{enumerate}
\end{thm}

The module $W_\mathbb F(\lambda)$ defined in this theorem is called the Weyl module of highest weight $\lambda$.

\begin{thm} Suppose $\mathbb F$ has characteristic zero. Then, every finite-dimensional $U_\mathbb F(\lie g)$-module is completely reducible. In particular, $W_\mathbb F(\lambda)$ is simple for all $\lambda\in P^+$.\hfill\qedsymbol
\end{thm}

\subsection{$U_\mathbb F(\tlie g)$-modules}

We now recall some basic results about the category of finite-dimensional $U_\mathbb F(\tlie g)$-modules in the same spirit of the previous section. All the results of this subsection and the next can be found in \cite{hyperlar} and references therein.

We denote by $\cal P_\mathbb F^+$ the multiplicative monoid consisting of all families of the form $\gb\omega = (\gb\omega_i)_{i\in I}$ where each $\gb\omega_i$ is a polynomial in $\mathbb F[u]$ with constant term $1$. We also denote by $\cal P_\mathbb F$ the multiplicative group associated to $\cal P_\mathbb F^+$ which will be referred to as the $\ell$-weight lattice associated to $\lie g$. Given $\mu\in P$ and $a\in\mathbb F^\times$, let $\gb\omega_{\mu,a}$ be the element of $\cal P_\mathbb F$ defined as
\begin{equation*}
(\gb\omega_{\mu,a})_i(u) = (1-au)^{\mu(h_i)} \quad\text{for all}\quad i\in I.
\end{equation*}
If $\mu=\omega_i$ is a fundamental weight, we simplify notation and write $\gb\omega_{i,a}$ and refer to it as a fundamental $\ell$-weight. Notice that $\cal P_\mathbb F$ is the free abelian group on the set of fundamental $\ell$-weights.
Let $\wt:\cal P_\mathbb F \to P$ be the unique group homomorphism such that $\wt(\gb\omega_{i,a}) = \omega_i$ for all $i\in I,a\in\mathbb F^\times$. Let also $\gb\omega\mapsto \gb\omega^-$ be the unique group automorphism of $\cal P_\mathbb F$ mapping $\gb\omega_{i,a}$ to $\gb\omega_{i,a^{-1}}$ for all $i\in I,a\in\mathbb F^\times$. For notational convenience we set $\gb\omega^+=\gb\omega$.

The abelian group $\cal P_\mathbb F$ can be identified with a subgroup of the monoid of $|I|$-tuples of formal power series with coefficients in $\mathbb F$ by identifying the rational function $(1-au)^{-1}$ with the corresponding geometric formal power series. This allows us to define an inclusion $\cal P_\mathbb F \hookrightarrow (U_\mathbb F(\tlie h))^*$   determined by
\begin{gather*}
\gb\omega\left(\tbinom{h_i}{k}\right) = \tbinom{\wt(\gb\omega)(h_i)}{k}, \quad \gb\omega(\Lambda_{i,r})=\omega_{i,r}, \quad\text{for all}\quad i\in I,r,k\in\mathbb Z,k\ge 0,\\ \text{and}\quad
\gb\omega(xy)=\gb\omega(x)\gb\omega(y), \quad\text{for all}\quad x,y\in U_\mathbb F(\tlie h).
\end{gather*}
Here, $\omega_{i,\pm r}$ is the coefficient of $u^r$ in the $i$-th formal power series of $\gb\omega^\pm$.

Given a $U_\mathbb F(\tlie g)$-module $V$ and $\xi\in U_\mathbb F(\tlie h)^*$, let
\begin{equation*}
V_\xi=\{v\in V: \text{ for all } x\in U_\mathbb F(\tlie h), \text{ there exists } k>0 \text{ such that } (x-\xi(x))^kv = 0\}.
\end{equation*}
We say that $V$ is an $\ell$-weight module if
\begin{equation*}
V = \opl_{\gb\omega\in\cal P_\mathbb F}^{} V_\gb\omega.
\end{equation*}
In that case, it follows that
\begin{equation*}
V_\mu=\opl_{\substack{\gb\omega\in\cal P_\mathbb F:\\ \wt(\gb\omega)=\mu}}^{} V_\gb\omega \quad\text{for all}\quad \mu\in P \quad\text{and}\quad V=\opl_{\mu\in P}^{} V_\mu.
\end{equation*}

A nonzero element of $V_\gb\omega$ is said to be an $\ell$-weight vector of $\ell$-weight $\gb\omega$. An $\ell$-weight vector $v$ is said to be a highest-$\ell$-weight vector if $U_\mathbb F(\tlie h)v=\mathbb Fv$ and $(x_{\alpha,r}^+)^{(k)}v = 0$ for all $\alpha\in R^+$ and all $r,k\in\mathbb Z, k>0$. If $V$ is generated by a highest-$\ell$-weight vector of $\ell$-weight $\gb\omega$, $V$ is said to be a highest-$\ell$-weight module of highest $\ell$-weight $\gb\omega$.

\begin{thm} Let $V$ be a $U_\mathbb F(\tlie g)$-module.
\begin{enumerate}[(a)]
\item If $V$ is finite-dimensional, then $V$ is an $\ell$-weight module. In particular, if $V$ is finite-dimensional and irreducible, then $V$ is a highest-$\ell$-weight module whose highest $\ell$-weight lies in $\cal P_\mathbb F^+$.
\item If $V$ is a highest-$\ell$-weight module of highest $\ell$-weight $\gb\omega\in\cal P_\mathbb F^+$, then $\dim V_\gb\omega=1$ and $V_{\mu}\ne 0$ only if $\mu\le \wt(\gb\omega)$. Moreover, $V$  has a unique maximal proper submodule and, hence, also a unique irreducible quotient. In particular, $V$ is indecomposable.

\item For each $\gb\omega\in \cal P_\mathbb F^+$, the $U_\mathbb F(\tlie g)$-module $W_\mathbb F(\gb\omega)$ generated by a vector
$v$ satisfying the defining relations of being a highest-$\ell$-weight vector of $\ell$-weight $\gb\omega$ and
\begin{gather*}
(x_{\alpha}^-)^{(l)}v = 0 \quad\text{for all}\quad \alpha\in R^+, l>\wt(\gb\omega)(h_\alpha),
\end{gather*}
is nonzero and finite-dimensional. Moreover, every finite-dimensional highest-$\ell$-weight-module of highest $\ell$-weight $\gb\omega$ is a quotient of $W_\mathbb F(\gb\omega)$.
\item If $V$ is finite-dimensional and irreducible, then there exists a unique $\gb\omega\in P^+$ such that $V$ is isomorphic to the irreducible quotient $V_\mathbb F(\gb\omega)$ of $W_\mathbb F(\gb\omega)$.
\item For $\mu\in P$ and $\gb\omega\in\cal P_\mathbb F^+$, we have $\mu\in\wt(W_\mathbb F(\gb\omega))$ if and only if $\mu\in\wt(W_\mathbb F(\wt(\gb\omega)))$.\hfill\qedsymbol
\end{enumerate}
\end{thm}

The module $W_\mathbb F(\gb\omega)$ defined above is called the Weyl module of highest $\ell$-weight $\gb\omega$.

\subsection{Evaluation modules}

For a $U_\mathbb F(\lie g)$-module $V$ and $a\in\mathbb F^\times$, denote by $V(a)$ the pullback of $V$ by ${\rm ev}_a$ (cf. Section \ref{s:evmaps}). The $U_\mathbb F(\lie g)$-modules constructed in this manner are referred to as evaluation modules. If $V=V_\mathbb F(\lambda)$ for some $\lambda\in P^+$, we shall denote the corresponding evaluation module by $V_\mathbb F(\lambda,a)$. It is not difficult to check using \eqref{e:ev(gen)} that, if $v$ is a weight vector of weight $\lambda$, then
\begin{equation}
\ev_a(\Lambda_i^+(u))\ v = (\gb\omega_{i,a}(u))^{\lambda(h_i)}\ v.
\end{equation}
In particular,
\begin{equation}
V_\mathbb F(\lambda,a)\cong V_\mathbb F(\gb\omega_{\lambda,a}).
\end{equation}

\begin{prop}\label{p:evmoduntw}
Let $\gb\omega = \prod_{j=1}^n \gb\omega_{\lambda_j,a_j}\in\cal P_\mathbb F^+$ with $a_i\ne a_j$ for $i\ne j$. Then, $V_\mathbb F(\gb\omega)\cong \otm_{j=1}^n V_\mathbb F (\lambda_j,a_j)$.\hfill\qedsymbol
\end{prop}

\section{Finite-dimensional $U_\mathbb F(\tlie g^\sigma)$-modules}\label{s:main}

In this section we start the study of finite-dimensional $U_\mathbb F(\tlie g^\sigma)$-modules. In particular, we develop the corresponding highest-$\ell$-weight theory, define the Weyl modules in terms of generators and relations, show that they have the same universal property as their non-twisted counterparts,  and establish the classification of the irreducible modules. In particular, we recover the related results from \cite{CFS} in the case that $\mathbb F$ has characteristic zero. The proofs are all based on their non-twisted analogues found in \cite{hyperlar}, but several extra details are needed.

\subsection{Highest-$\ell$-weight modules}\label{ss:tlw}

Recall from Section \ref{s:simple} that $P_0$ is the dominant weight lattice of $\lie g_0$. Let $P^{\sigma}_0$ be the subset of $P_0$ defined as follows:
\begin{equation*}
    P^{\sigma}_0  = \begin{cases}  \mu \in P_0 \text{ such that } \mu(h_{i,0}) \in \mathbb{Z}, & \text{if } \lie g \text{ is of type } A_{2n} \text{ and } \alpha_i\in R_s\\
     P_0, & \text{otherwise.} \end{cases}
\end{equation*}
Set $P_0^{\sigma,+}=P_0^\sigma\cap P_0^+$. We shall also regard each $\mu\in P_0^{\sigma}$ as an element of $P$ by setting
\begin{equation}\label{extending}
    \mu(h_i)=\begin{cases} \mu(h_{i,0}),  & \text{ if } i\in o(I_0), \\ 0,   & \text{ otherwise.} \end{cases}
\end{equation}

We denote by $\cal P_\mathbb F^{\sigma,+}$ the multiplicative monoid consisting of all families of the form $\gb\omega = (\gb\omega_i)_{i\in I_0}$ where each $\gb\omega_i$ is a polynomial in $\mathbb F[u]$ with constant term $1$. We also denote by $\cal P_\mathbb F^{\sigma}$ the multiplicative group associated to $\cal P_\mathbb F^{\sigma,+}$ which will be referred to as the $\ell$-weight lattice associated to the pair $(\lie g,\sigma)$.  For $i\in I_0$, $a\in\mathbb F^\times$, and $\mu\in P_0^{\sigma}$, define the element $\gb \omega^\sigma_{\mu,a}$ in $\cal P_\mathbb F^{\sigma}$ whose $i$-th entry is
\begin{equation}\label{drinf}
    (\gb \omega^\sigma_{\mu,a})_i(u)=
   \begin{cases}
   (1-au)^{\mu(h_{i,0})}, & \text{if either $\lie g$ is of type $A_{2n}$, or if $\lie g$ is not of type $A_{2n}$ and $\alpha_i\in R_s$}, \\
   (1-a^mu)^{\mu(h_{i,0})}, & \text{if $\lie g$ is not of type $A_{2n}$ and $\alpha_i\in R_l$}.
    \end{cases}
\end{equation}
As in the non-twisted case, we simplify notation and write $\gb\omega_{i,a}^\sigma$ in place of $\gb\omega_{\omega_i,a}^\sigma$ and refer to it as a fundamental $\ell$-weight. Notice that $\cal P_\mathbb F^\sigma$ is the free abelian group on the set of fundamental $\ell$-weights. Indeed, the map $\mathbb F\to\mathbb F$ given by $a\mapsto a^k$ for $k=1,2,3$ is surjective since $\mathbb F$ is algebraically closed and, hence, a perfect field.

Let $\wt:\cal P_\mathbb F^{\sigma}\to P_0^{\sigma}$ be the unique group homomorphism such that
\begin{gather}
   \wt({\gb \omega^\sigma_{i,a}})=\begin{cases}
    2\omega_i, & \text{ if $\lie g$ is  of type $A_{2n}$ and $\alpha_i\in R_s$}, \\
    \omega_i, & \text{otherwise.} \end{cases}
\end{gather}
As in the non-twisted setting, let $\gb\omega\mapsto \gb\omega^-$ be the unique group automorphism of $\cal P_\mathbb F^\sigma$ mapping $\gb\omega_{i,a}^\sigma$ to $\gb\omega_{i,a^{-1}}^\sigma$ for all $i\in I,a\in\mathbb F^\times$. As before, for notational convenience, we set $\gb\omega^+=\gb\omega$. Also, the abelian group $\cal P_\mathbb F^\sigma$ can be identified with a subgroup of the monoid of $|I_0|$-tuples of formal power series with coefficients in $\mathbb F$ as before. We then define an inclusion $\cal P_\mathbb F^\sigma \hookrightarrow U_\mathbb F(\tlie h^\sigma)^*$  by setting
\begin{gather*}
\gb\omega\left(\tbinom{h_{i,0}}{k}\right) = \tbinom{\wt(\gb\omega)(h_{i,0})}{k}, \quad  \gb\omega(\Lambda_{i,r}^{\sigma,\pm}) = \omega_{i,r} \quad\text{for all}\quad i\in I_0,r,k\in\mathbb Z,k\ge 0,\\ \quad\text{and}\quad\\
\gb\omega(xy)=\gb\omega(x)\gb\omega(y) \quad\text{for all}\quad x,y\in U_\mathbb F(\tlie h^\sigma),
\end{gather*}
where $\omega_{i,\pm r}$ is the coefficient of $u^r$ in the $i$-th formal power series of $\gb\omega^\pm$.

Given a $U_\mathbb F(\tlie g^\sigma)$-module $V$ and $\xi\in U_\mathbb F(\tlie h^\sigma)^*$, let
\begin{equation*}
V_\xi=\{v\in V: \text{ for all } x\in U_\mathbb F(\tlie h^\sigma), \text{ there exists } k>0 \text{ such that } (x-\xi(x))^kv = 0\}.
\end{equation*}
We say that $V$ is an $\ell$-weight module if
\begin{equation*}
V = \opl_{\gb\omega\in\cal P_\mathbb F^\sigma}^{} V_\gb\omega.
\end{equation*}
In that case, it follows that
\begin{equation*}
V_\mu=\opl_{\substack{\gb\omega\in\cal P_\mathbb F^\sigma:\\ \wt(\gb\omega)=\mu}}^{} V_\gb\omega \quad\text{for all}\quad \mu\in P_0 \quad\text{and}\quad V=\opl_{\mu\in P_0}^{} V_\mu.
\end{equation*}
Moreover, by \eqref{comutxhtw}, we have that $$(x^\pm_{\mu,-r}\otimes t^r)^{(k)} V_\mu \subseteq V_{\mu \pm k\mu} \quad \text{for all}\quad x^\pm_{\mu,-r}\in\cal C^\sigma(O), k\in \mathbb Z_+.$$
A nonzero element of $V_\gb\omega$ is said to be an $\ell$-weight vector of $\ell$-weight $\gb\omega$. An $\ell$-weight vector $v$ is said to be a highest-$\ell$-weight vector if $U_\mathbb F(\tlie h^\sigma)v=\mathbb Fv$ and $U_\mathbb F(\tlie g^\sigma )^0v = 0$ . If $V$ is generated by a highest-$\ell$-weight vector of $\ell$-weight $\gb\omega$, $V$ is said to be a highest-$\ell$-weight module of highest $\ell$-weight $\gb\omega$.
Standard arguments show that:

\begin{prop} \label{uniqirred} Every highest-$\ell$-weight module has a unique proper maximal submodule and, hence, a unique irreducible quotient.\hfill\qedsymbol
\end{prop}

We shall denote by $V_\mathbb F(\gb\omega)$ the irreducible quotient of a highest-$\ell$-weight module of highest $\ell$-weight $\gb\omega\in\cal P_\mathbb F^\sigma$.

The next proposition, which is the ``twisted version'' of \cite[Proposition 3.1]{hyperlar}, establishes a set of relations satisfied by all finite-dimensional highest-$\ell$-weight modules.

\begin{prop}\label{ellhwreltw}
Let $V$ be a finite-dimensional $U_\mathbb F(\tlie g^\sigma)$-module, $\lambda\in P_0^{+}$,  and $v\in V_\lambda$ such that
$U_\mathbb F((\tlie n^+)^\sigma)^0$ $v= 0$ and $\Lambda_{i,r}^\sigma v = \omega_{i,r}v$, for all $i\in I_0, r\in\mathbb Z$, and some $\omega_{i,r}\in \mathbb F$. Then:
\begin{enumerate}[(a)]
   \item $\lambda\in P_0^{\sigma,+}$;
    \item $(x^-_{\mu,0} \otimes t^{ms})^{(k)}v =0$ for all $x^-_{\mu,0}\in\cal C^\sigma(O), k>d_\mu\lambda(h_{\mu,0})$, where $d_\mu=2$ if $\lie g$ is of type $A_{2n}$ and $\mu \in R_s$ and  $d_\mu=1$ otherwise;

    \item $\Lambda_{i,\pm r}^\sigma v=0 \text{ for all } i\in I_0, r>\lambda(h_{i,0})$;

    \item $\omega_{i,\pm  \lambda(h_i)} \ne 0$.
\end{enumerate}
Moreover, for all $\lambda\in P_0^{\sigma,+}$, there exist polynomials $f_{i,r}\in\mathbb F[t_0,t_1,\cdots,t_{\lambda(h_{i,0})}], i\in I_0, r=1,\cdots, \lambda(h_{i,0})$, such that for all $V$ and $v$ as above, we have $\omega_{i,-r}v = f_{i,r}(\omega_{i,\lambda(h_{i,0})}^{-1}, \omega_{i,1},\cdots, \omega_{i,\lambda(h_{i,0})})v$.
\end{prop}

\proof Part (a) is immediate when $\lie g$ is not of type $A_{2n}$. If $\lie g$ is of type $A_{2n}$, given $\mu \in R_s$, the subalgebra  $U_\mathbb F(\tlie g_{2\mu,1})$  of $U_\mathbb F(\tlie g^\sigma)$ generated by $\{ \tbinom{h_{\mu,0}}{k}, (x^\pm_{2\mu,1}\otimes t^{\mp 1})^{(k)}:k\ge 0\}$ is isomorphic to $U_\mathbb F(\lie{sl}_2)$ by Lemma \ref{[]tw}(b). Hence,  $\lambda(h_{\mu,0}) \in \mathbb Z$.

For part (b), observe that Lemma \ref{[]tw} implies that, if $r\in\mathbb Z$ and $\mu \in \wt(\lie g_{\pm r})\cap Q_0^+\setminus\{0\}$, the elements $(x^\pm_{\mu,\pm r} \otimes t^{\pm r})^{(k)}$,  $k\in\mathbb Z_+$, generate a subalgebra $U_\mathbb F(\tlie g_{\mu,r})$  of $U_\mathbb F(\tlie g^\sigma)$ isomorphic to $U_\mathbb F(\lie{sl}_2)$. Hence, the equality $(x^-_{\mu,0} \otimes t^{ms})^{(k)}v = 0$ in each case follows from the fact that $v$ generates a finite-dimensional highest-weight module for this subalgebra, which is then isomorphic to a quotient of the Weyl module $W_\mathbb F(d_\mu\lambda(h_{\mu,0}))$.

Further, if either $\lie g$ is not of type $A_{2n}$ and $\alpha_i \in R_s$, or  $\lie g$ is of type $A_{2n}$ and $\alpha_i \in R_l$,  we conclude that $\Lambda_{i,\pm r}^\sigma v=0$ for $r>|\lambda(h_{i,0})|$ by setting $\alpha=\alpha_i,s=0,l=k=r$ in Lemma \ref{basicreltw}(a). Similarly, application of Lemma \ref{basicreltw}(b) proves the claim for $\lie g$ not of type $A_{2n}$ and $\alpha_i\in R_l$.  For the case that $\lie g$ is of type $A_{2n}$ and $\alpha_i\in R_s$, we conclude that $\Lambda_{i,\pm r}^\sigma v=0$ for $r>|\lambda(h_{i,0})|$ by using part (i) of Lemma \ref{basicreltw}(c) with $\alpha=\alpha_i,s=a=0,k=r$, along with the already used observation that $U_\mathbb F(\tlie g_{2\mu,1})$ is isomorphic to $U_\mathbb F(\lie{sl}_2)$. This completes the proof of (c).

We now prove (d). Let $W=U_\mathbb F(\tlie g^\sigma) v$ which is a finite-dimensional $U_\mathbb F(\lie g_0)$-module having $\lambda$ as its highest-weight. Hence, $\wt(W) = \wt(W_\mathbb F(\lambda))$ and Theorem \ref{t:rh}(e) then implies that
\begin{equation}\label{e:ellhwreltw}
\lambda-(\lambda(h_{i,0})+k)\alpha_i \notin\wt(W) \quad\text{for any}\quad k>0.
\end{equation}
We now split the proof in cases according with the conditions given by each item of the Lemma \ref{basicreltw}. Suppose first that either $\lie g$ is not of type $A_{2n}$ and $\alpha_i \in R_s$, or  $\lie g$ is of type $A_{2n}$ and $\alpha_i \in R_l$.  By considering the subalgebra  $U_\mathbb F(\tlie g_{\alpha_i,0})\cong U_\mathbb F(\lie{sl}_2)$, we conclude that $(x_{i,0}^-\otimes 1)^{(\lambda(h_{i,0}))}v\ne 0$. It follows from \eqref{e:ellhwreltw} that $(x^-_{i,-1}\otimes t)^{(k)}(x^-_{i,0}\otimes 1)^{(\lambda(h_{i,0}))}v=0$ for all $k>0$. Therefore, $(x_{i,0}^-\otimes 1)^{(\lambda(h_{i,0}))}v$ generates a lowest-weight finite-dimensional representation of $U_\mathbb F(\tlie g_{\alpha_i,1})\cong U_\mathbb F(\lie{sl}_2)$. This implies that $0\ne(x_{i,-1}^+ \otimes t)^{(\lambda(h_{i,0}))}(x_{i,0}^-\otimes 1)^{(\lambda(h_{i,0}))}v=\Lambda_{i,\pm \lambda( h_{i,0})}^\sigma v$, where the last equality follows from Lemma \ref{basicreltw}(a).  Assume now that $\lie g$ is not of type $A_{2n}$ and $\alpha_i \in R_l$. Proceeding similarly we conclude that $(x_{i,0}^-\otimes 1)^{(\lambda(h_{i,0}))}v\ne 0$ and generates a lowest-weight finite-dimensional representation of $U_\mathbb F(\tlie g_{\alpha_i,m})$. The conclusion now follows from Lemma \ref{basicreltw}(b) in a similar fashion. Finally, let $\lie g$ be of type $A_{2n}$ and $\alpha_i \in R_s$. It now follows that $(x_{2\alpha_i,1}^-\otimes t)^{(\lambda(h_{i,0}))}v\ne 0$ generates a lowest-weight finite-dimensional representation of $U_\mathbb F(\tlie g_{\alpha_i,0})$.
A similar application of part (i) of Lemma \ref{basicreltw}(c) completes the proof.

The proof of the last statement is also split in cases fitting the conditions of Lemma \ref{basicreltw}. Assume first that either $\lie g$ is not of type $A_{2n}$ and $\alpha_i \in R_s$, or $\lie g$ is of type $A_{2n}$ and $\alpha_i \in R_l$. In this case, setting $\alpha=\alpha_i, s=0,l=\lambda(h_{i,0})$ and $k=l+r$ in Lemma \ref{basicreltw} (a), we get for all $r\ge1$ that
\begin{equation}\label{e:l1}
0=(x_{i,0}^+\otimes 1)^{(l)}(x_{i,-1}^-\otimes t)^{(k)} v = \omega_{i,l}(x^-_{i,-1}\otimes t)^{(r)}v + \tsum_{j=1}^{l} \omega_{i,l-j}Y_jv=0,
\end{equation}
where $\omega_{i,0}=1$ and $Y_j$ is a $\mathbb Z$-linear combination of elements of the form
$$(x^-_{i,-1} \otimes t)^{(k_1)}\cdots (x^-_{i,-(r+1)} \otimes t^{r+1})^{(k_{r+1})}$$
with $\sum_n k_n = r$ and $\sum_n nk_n = r+j$, which does not depend neither on $V$ nor on $v$. Now, since $-r<r+j-2r$, it is not difficult to see that $(x^+_{i,2}\otimes t^{-2})^{(r)}Y_j \in U_\mathbb F(\tlie g^\sigma)U_\mathbb F((\tlie n^+)^\sigma)^0 + H_{r,j}$, where $H_{r,j}$ is a linear combination of monomials of the form $\Lambda_{i,r_1}^\sigma\cdots \Lambda_{i,r_m}^\sigma$ such that $-r<r_j$. Moreover,  using Lemma \ref{basicreltw}(a) once more, we get
$$(x_{i,2}^+\otimes t^{-2})^{(r)}(x_{i,-1}^-\otimes t^{})^{(r)} = (-1)^r\Lambda_{i,-r}^\sigma + U_\mathbb F(\tlie g)U_\mathbb F(\tlie n^+)^0.$$
Hence, plugging this into \eqref{e:l1}, it follows that
\begin{equation*}
    0=(x^+_{i,2}\otimes t^{-2})^{(r)}\left(\omega_{i,l}(x^-_{i,-1}\otimes t)^{(r)}v + \tsum_{j=1}^{l} \omega_{i,l-j}Y_jv\right)= (-1)^r\omega_{i,-r}\omega_{i,l}v + \tsum_{j=1}^{l} \omega_{i,l-j}H_{r,j}v
\end{equation*}
which implies that
\begin{equation}\label{indonr}
    \omega_{i,-r}\omega_{i,l}v = (-1)^r\tsum_{j=1}^{r} \omega_{i,l-j}H_{r,j}v.
\end{equation}
Now, since $H_{r,j}$ involves only the elements $\Lambda_{i,s_j}$ with $s_j>-r$, a simple induction on $r$ completes the proof in this case. Similarly, when  $\lie g$ is not of type $A_{2n}$ and $\alpha_i\in R_l$, we repeat this argument using Lemma \ref{basicreltw}(b). Finally, if $\lie g$ is of type $A_{2n}$ and $\alpha_i \in R_s$, setting $\alpha=\alpha_i$ and $k=\lambda(h_{i,0})$ on part (ii) of Lemma \ref{basicreltw}(c), we get for all $r\ge 1$ that
\begin{equation}\label{e:l2}
0=(x^-_{\alpha_i,0}\otimes 1)^{(2k)}(x^-_{2\alpha_i,1}\otimes t^{})^{(k+r)}v=\left(\omega_{i,k}(x^-_{2\alpha_i,1}\otimes t^{})^{(r)}+ \tsum_{j=1}^{k-1} \omega_{i,k-j} X_{j}\right) v
\end{equation}
where $\omega_{i,0}=1$ and  $X_{j}$ is a $\mathbb Z$-linear combination of elements of the form
$$(x^-_{\alpha_i,1} \otimes t^{})^{(s_1)}\cdots (x^-_{\alpha_i,-k} \otimes t^{k})^{(s_{k})}(x^-_{2\alpha_i,1} \otimes t^{})^{(s_1')}\cdots (x^-_{2\alpha_i,-k} \otimes t^{k})^{(s_{k}')}$$
with $s_i'<r$, $\sum_m s_m+2\sum_n s_n' = 2r$ and $\sum_m ms_m'+\sum_n ns_n = r+j$, which does not depend neither on $V$ nor on $v$. Proceeding as above, since $-r<r+j-2r$, we first conclude that
$$(x^+_{\alpha_i,1}\otimes t^{-1} )^{(2r)}X_{j} \in U_\mathbb F(\tlie g^\sigma)U_\mathbb F((\tlie n^+)^\sigma)^0 + H_{r,j}$$
where $H_{r,j}$ is a linear combination of monomials of the form $\Lambda_{i,r_1}^\sigma\cdots \Lambda_{i,r_m}^\sigma$ with $-r<r_j$. Then, we get from part (i) of Lemma \ref{basicreltw}(c) that
$$(x^+_{\alpha_i,1}\otimes t^{-1})^{(2r)} (x^-_{2\alpha_i,1}\otimes t^{})^{(r)} = -\Lambda_{i,-r}^{\sigma} + U_\mathbb F(\tlie g)U_\mathbb F(\tlie n^+)^0.$$ Hence, plugging this into \eqref{e:l2}, we get
$$0=(x^+_{\alpha_i,1}\otimes t^{-1})^{(2r)}\left(  \omega_{i,k}(x^-_{2\alpha_i,1}\otimes t^{})^{(r)}v+ \tsum_{j=1}^{k-1} \omega_{i,k-j} X_{j}v \right)= -\omega_{i,-r}\omega_{i,k}v + \tsum_{j=1}^{k-1} \omega_{i,k-j}H_{r,j}v$$
and the proof ends by induction on $r$ as it was done following \eqref{indonr}.\endproof

As in the non-twisted case, given $v$ as in the above proposition, we want to prove that
\begin{equation}\label{Lambdaonv}
\Lambda^\sigma_{i,\lambda(h_{i,0})}\Lambda^\sigma_{i,-k}v = \Lambda^\sigma_{i,\lambda(h_{i,0})-k }v \qquad\text{for all} \qquad i \in I_0,\ 0 \leq k \leq  \lambda(h_{i,0}).
\end{equation}
In other words, given $v, \lambda$, $i\in I_0$, and $\omega_{i,r}$ as in the proposition and setting $$\gb\omega_{i}(u) = 1+\tsum_{r=1}^{\lambda(h_{i,0})}  \omega_{i,r}u^r,$$ we want to show that
\begin{equation}\label{drinfeldaction}
    \Lambda_{i}^{\sigma,-}(u)v = \gb\omega_{i}^-(u)v.
\end{equation}
The element $\gb\omega:=(\gb\omega_{i})_{i\in I_0}$ is called the Drinfeld polynomial of the highest $\ell$-weight module generated by $v$.
As observed in \cite{hyperlar}, the differential equations techniques used in \cite{CPweyl,CFS} for proving \eqref{Lambdaonv} do not work in positive characteristic. Thus, in light of the last statement of the Proposition \ref{ellhwreltw}, it suffices to exhibit for each $\gb\omega \in \cal P_\mathbb F^{\sigma,+}$ one finite-dimensional highest-$\ell$-weight module with highest-$\ell$-weight $\gb\omega$ on which \eqref{drinfeldaction} is satisfied. This will be done in the next section.

\subsection{Evaluation modules}\label{ss:evm}

If $V$ is a finite-dimensional irreducible $U_\mathbb F(\tlie g^\sigma)$-module, then $V$ is generated by a vector $v$ satisfying
$$U_\mathbb F((\tlie n^+)^\sigma)^0 v =  0,\ \quad \ \tbinom{h_{i,0}}{k}v = \tbinom{\lambda(h_{i,0})}{k}v,\ \quad  \ \Lambda_{i,r}^\sigma v = \omega_{i,r}v,$$
for all $i \in I_0, k\in\mathbb Z_+$ and some $\lambda\in P_0^{\sigma,+}$, $\omega_{i,r}\in \mathbb F$. In fact, since $V$ is finite-dimensional, there exists a maximal weight $\lambda \in P_0^+$ such that $V_\lambda \ne 0$. Moreover, as $U_\mathbb F(\tlie h^\sigma)$ is commutative, we  conclude that $\Lambda_{i,r}^\sigma V_\lambda \subseteq V_\lambda$ for all $r\in \mathbb Z,i \in I_0$. It follows that there exist $v\in V_\lambda$ satisfying $U_\mathbb F(\tlie h^\sigma)v=\mathbb F v$. Hence, by Proposition \ref{ellhwreltw}(a), we have that $\lambda\in P_0^{\sigma,+}$ and the $U_\mathbb F(\tlie g^\sigma)$-submodule generated by $v$ must coincide with $V$ by the irreducibility of $V$. In particular, we have the following immediate consequence of Proposition \ref{ellhwreltw}.

\begin{prop}\label{p:smhw}
Every finite-dimensional irreducible $U_\mathbb F(\tlie g^\sigma)$-module is a highest-$\ell$-weight module whose highest $\ell$-weight lies in $ \cal P_\mathbb F^{\sigma,+}$. \hfill\qedsymbol
\end{prop}

For a $U_\mathbb F(\lie g)$-module $V$, let $V(a)$ be the pull-back of $V$ by ${\rm ev}_a^\sigma$ (cf. Section \ref{s:evmaps}). In the cases $V=V_\mathbb F(\lambda)$ and $V=W_\mathbb F(\lambda)$, we shall denote the evaluation representation by $V_\mathbb F(\lambda,a)$ and by $W_\mathbb F(\lambda,a)$, respectively.

\begin{prop}\label{p:evmod}
Let $\lambda\in P_0^{\sigma,+}$ and regard it as an element of $P^+$ as in \eqref{extending}. Then, given $a\in\mathbb F^\times$ and a highest-weight $U_\mathbb F(\lie g)$-module  $V$ of highest-weight $\lambda$, $V(a)$ is a highest $\ell$-weight $U_\mathbb F(\tlie g^\sigma)$-module with Drinfeld polynomial $\gb\omega_{\lambda,a}^\sigma \in\cal P_\mathbb F^{\sigma,+}$. Moreover, the action of $\Lambda_i^{\sigma,-}(u)$ on the highest $\ell$-weight vector is given by \eqref{drinfeldaction} for all $i\in I_0$.
\end{prop}

\proof Let $v$ be a highest-weight vector of $V$. Suppose first that $o(i)$ is not fixed by $\sigma$. Then, by \eqref{e:evsigma},
\begin{align*}
\Lambda_i^{\sigma,\pm}(u)\ v  &= \left( \tsum_{r=0}^\infty  (-a)^{\pm r} \tsum_{\substack{ r_0,\cdots, r_{m-1} \in \mathbb Z_+ \\ r_0+\cdots+r_{m-1}= r}}^{} \tprod_{j=0}^{m-1} \zeta^{-jr_j}\tbinom{h_{\sigma^j(\alpha_{o(i)})}}{r_j} u^r\right) \ v = \left(\tsum_{r=0}^\infty (-a)^{\pm r} \tbinom{\lambda(h_{\alpha_{o(i)}})}{r} u^r\right) \ v\\ & = (1-a^{\pm 1}u)^{\lambda(h_{\alpha_{o(i)}})}\ v,
\end{align*}
where the second equality follows since $\lambda( h_{\sigma^j(\alpha_{o(i)})})=0$ for $0< j \le m-1$.
Similarly, if $o(i)$ is fixed by $\sigma$, then, again by \eqref{e:evsigma}, we have
\begin{align*}
\Lambda_i^{\sigma,\pm}(u)\ v  &= \left(\tsum_{r=0}^\infty (-a^m)^{\pm r} \tbinom{h_{\alpha_{o(i)}}}{r} u^r\right) \ v = \left(\tsum_{r=0}^\infty (-a^m)^{\pm r} \tbinom{\lambda(h_{\alpha_{o(i)}})}{r} u^r\right) \ v  = (1-a^{\pm m}u)^{\lambda(h_{\alpha_{o(i)}})}\ v.
\end{align*}
\endproof

 Now we can conclude the proof of \eqref{drinfeldaction}. Given $\gb\omega\in\cal P_\mathbb F^{\sigma,+}$, we can write
\begin{equation}\label{standard}
\gb\omega=\tprod_{k=1}^\ell\tprod_{\epsilon=0}^{m-1}\gb\omega^\sigma_{\lambda_{k,\epsilon},\zeta^{m-\epsilon} a_k}, \text{ with } \lambda_{k,\epsilon}\in P_0^+ \text{ and } a_k \in \mathbb F^\times \text{ such that } a_i^m \ne a_j^m \text{ for } i\ne j.
\end{equation}
Any such expression is called a standard decomposition of $\gb\omega$. Let $T_\mathbb F(\gb\omega)$ be the submodule of
$$\otm_{k=1}^\ell \otm_{\epsilon=0}^{m-1} V_\mathbb F(\gb\omega_{\lambda_{k,\epsilon},\zeta^\epsilon a_k}^\sigma)$$
generated by the tensor product of the highest weight vectors. Then, $T_\mathbb F(\gb\omega)$ is a  nonzero finite-dimen\-sional highest $\ell$-weight module and it follows from \eqref{e:comutLambdaser} and Proposition \ref{p:evmod} that its highest $\ell$-weight is $\gb\omega$. This completes the proof of \eqref{drinfeldaction}.

\begin{thm}\label{t:csm}
The isomorphism classes of finite-dimensional simple $U_\mathbb F(\tlie g^\sigma)$-modules is in bijection with $\cal P^{\sigma,+}$.
\end{thm}

\begin{proof}
By Proposition \ref{p:smhw}, to each isomorphism class of finite-dimensional simple $U_\mathbb F(\tlie g^\sigma)$-modules we have associated an element of $\cal P^{\sigma,+}$. Thus, it remains to see that this map is surjective. Indeed, given $\gb\omega\in \cal P^{\sigma,+}$, the module $T_\mathbb F(\gb\omega)$ constructed above is a finite-dimensional highest-$\ell$-module of highest $\ell$-weight $\gb\omega$. Hence, it has a simple quotient by Proposition \ref{uniqirred}.
\end{proof}

We shall denote by $V_\mathbb F(\gb\omega)$ any simple finite-dimensional simple $U_\mathbb F(\tlie g^\sigma)$-module with highest $\ell$-weight $\gb\omega$.

\subsection{The Weyl modules}\label{ss:wm}
We now establish the existence of universal finite-dimensional highest $\ell$-weight $U_\mathbb F(\tlie g^\sigma)$-modules.

\begin{defn}
Given $\gb\omega \in \cal P_\mathbb F^{\sigma,+}$, let $W_\mathbb F(\gb\omega)$ be the $U_\mathbb F(\tlie g^\sigma)$-module generated by a vector $v$ satisfying the defining relations of being a highest-$\ell$-weight vector of $\ell$-weight $\gb\omega$ and
\begin{gather}\label{weylrel}
(x_{\mu, 0}^- \otimes t^{mr})^{(l)}v = 0,
\end{gather}
for all $\mu \in R_0^+, l,r\in\mathbb Z, l>\wt(\gb\omega)(h_{\mu,0})$.
$W_{\mathbb F}(\gb\omega)$ is called the Weyl module of highest $\ell$-weight $\gb\omega$.
\end{defn}

It follows from the last subsection that $W_{\mathbb F}(\gb\omega)\ne 0$  for all $\gb\omega\in\cal P_\mathbb F^{\sigma,+}$ (since $V_{\mathbb F}(\gb\omega)$ is clearly a nonzero quotient of $W_{\mathbb F}(\gb\omega)$).
Moreover, Proposition \ref{ellhwreltw} implies that every finite-dimensional highest-$\ell$-weight module of highest $\ell$-weight $\gb\omega\in \cal P_\mathbb F^{\sigma,\pm}$ is  a quotient of $W_\mathbb F(\gb\omega)$. Our final goal of this section is to prove that $W_{\mathbb F}(\gb\omega)$ is finite-dimensional.

\begin{lem} Let $\gb\omega\in\cal P_\mathbb F^{\sigma,\pm}$ and $\mu\in P_0$. If $W_\mathbb F(\gb\omega)_\mu\ne 0$, then $W(\gb\omega)_{w\mu}\ne 0$ for all $w\in \cal W_0$.
\end{lem}
\proof By standard arguments, it follows from \eqref{weylrel} that every vector $w\in W_\mathbb F(\gb\omega)$ belongs to a finite-dimensional $U_\mathbb F(\lie g_0)$-submodule of $W_\mathbb F(\gb\omega)$. Now all the claims follow from the corresponding results for finite-dimensional $U_\mathbb F(\lie g_0)$-modules.
\endproof

Notice that the previous lemma implies that
\begin{equation}\label{e:wtaW}
\wt(W_\mathbb F(\gb\omega)) \subseteq\wt(W_\mathbb F(\wt(\gb\omega))).
\end{equation}

We are ready to prove the next theorem which implies that the Weyl modules are the universal finite-dimensional highest-$\ell$-weight modules.

\begin{thm} \label{t:twfd}
$ W_{\mathbb F}(\gb\omega)$ is finite-dimensional for all $\gb\omega\in\cal P_\mathbb F^{\sigma,+}$.
\end{thm}

\proof
Set $\lambda = \wt(\gb\omega)$ and let $v$ be a highest-$\ell$-weight vector generating $W_{\mathbb F}(\gb\omega)$. Since \eqref{e:wtaW} implies that $\wt(W_\mathbb F(\gb\omega))$ is finite, it suffices to prove that $ W_{\mathbb F}(\gb\omega)$ is spanned by the elements
$$(x_{\mu_1,-s_1}^-\otimes t^{s_1})^{(k_1)}\cdots (x_{\mu_n,-s_n}^- \otimes t^{s_n})^{(k_n)}v,$$
with $n,k_j\in\mathbb Z_+,s_j\in\mathbb Z, \mu_j\in\wt(\lie g_{-s_j})\setminus\{0\}\cap Q_0^+, 0\le s_j< \lambda(h_{\mu_j,0}), j=1,\dots,n$.
Notice that such elements with no restriction on $s_j$  clearly span $ W_{\mathbb F}(\gb\omega)$.

Let $\cal T = Q_0^+\times \mathbb Z\times\mathbb Z_+$ and $\Xi$ be the set of functions $\phi:\mathbb N\to \cal T, j\mapsto \phi_j=(\mu_j,s_j,k_j)$, such that $k_j=0$ for all $j$ sufficiently large and $\mu_j\in\wt(\lie g_{-s_j})\setminus\{0\}$ for all $j$. Let also $\Xi'$ be the subset of $\Xi$ consisting of the elements $\phi$ such that $0\le s_j<\lambda(h_{\mu_j,0})$. Given $\phi\in\Xi$ such that $k_j=0$ for $j>m$, we associate the element
\begin{equation}\label{e:vphi}
v_\phi=  (x_{\mu_1,-s_1}^- \otimes t^{s_1})^{(k_1)}\cdots (x_{\mu_n,-s_n}^-\otimes t^{s_n})^{(k_n)}v  \in  W_{\mathbb F}(\gb\omega^\sigma).
\end{equation}
Define the degree of $\phi$ to be $d(\phi) = \sum_j k_j$ and the maximal exponent of $\phi$ to be $e(\phi)= \max\{k_j\}$.  Clearly $e(\phi)\le d(\phi)$ and $d(\phi)\ne 0$ implies $e(\phi)\ne 0$. Since there is nothing to be proved when $d(\phi)=0$ we assume from now on that  $d(\phi)>0$. Set
\begin{equation*}
\Xi_{d,e} = \{\phi\in \Xi: d(\phi)=d \text{ and } e(\phi)=e\} \qquad\text{and}\qquad \Xi_d=\bigcup\limits_{1\le e\le d} \Xi_{d,e}.
\end{equation*}
We prove by induction on $d$ and sub-induction on $e$ that, if $\phi\in\Xi_{d,e}$ is such that there exists $j$ with $s_j<0$ or $s_{j}\ge\lambda(h_{\mu_{j},0})$, then $v_\phi$ is in the span of vectors associated to elements in $\Xi'$. More precisely, given $0<e\le d\in\mathbb N$, we assume, by induction hypothesis, that this statement is true for every $\phi$ which belongs either to $\Xi_{d,e'}$ with $e'<e$ or to $\Xi_{d'}$ with $d'<d$. The proof is split in two cases according to whether $e=d$ or $e<d$.

When $e=d$, it follows that $v_\phi= (x_{\mu,-s}^- \otimes t^s)^{(e)}v$ for some $\mu$ and $s$. Suppose first that $e=1$. We split the proof in cases fitting the conditions of Lemma \ref{basicreltw}.
If either $\lie g$ is not of type $A_{2n}$ and $\mu \in R_s$, or if $\lie g$ is of type $A_{2n}$ and $\mu \in R_l$, setting $l=\lambda(h_{\mu, 0})$ and $k=l+1$ on Lemma \ref{basicreltw}(a), we get
\begin{equation*}
    \left( ( X^-_{\mu;r,\pm}(u))^{}\Lambda_{\mu}^{\sigma,\pm}(u)\right)_{\lambda(h_{\mu, 0})+1} v=0 \qquad\text{for all}\qquad  r\in\mathbb Z
\end{equation*}
or, equivalently,
\begin{equation}\label{1e=1}
(x^-_{\mu,-(r+1)}\otimes t^{r+1}\Lambda_{\mu,\pm l}^{\sigma}+x^-_{\mu,-(r+2)} \otimes t^{r+2}\Lambda_{\mu,\pm (l-1)}^{\sigma}+\cdots+ x^-_{\mu,-(r+l+1)} \otimes t^{r+l+1})v=0
\end{equation}
for all $r\in\mathbb Z$. Now, we consider the cases $s\ge l$ and $s<0$ separately and proceed with a further induction on $s$ and $|s|$, respectively. If $s\ge l$, this is easily done by taking $r=s-l-1$ in \eqref{1e=1}. Observing that $\Lambda_{\mu,\pm l}^{\sigma}v\ne 0$ (Lemma \ref{ellhwreltw}(c)), the case $s<0$ is dealt with by taking $r=s-1$ in \eqref{1e=1}.
Proceeding similarly, if $\lie g$ is not of type $A_{2n}$ and $\mu_i \in R_l$, we repeat this argument using  Lemma \ref{basicreltw}(b), and, if $\lie g$ is of type $A_{2n}$ and $\mu \in R_s$ or $\mu\in 2R_s$, the conclusion follows from parts (i) and (iii) of Lemma \ref{basicreltw}(c). We omit the details. This concludes the case $e=1$.
Suppose now that $e>1$. If either $\lie g$ is not of type $A_{2n}$ and $\mu \in R_s$, or if $\lie g$ is of type $A_{2n}$ and $\mu \in R_l$, setting $l=e\lambda(h_{\mu,0})$ and $k=l+e$ in Lemma \ref{basicreltw}(a) we obtain
\begin{equation}\label{1e>1}
\tsum_{j=1}^{\lambda(h_{\mu,0})} (x_{\mu,-(r+j+1)}^-\otimes t^{r+j+1})^{(e)}\Lambda_{\mu,l-ej}v + \substack{\text{ other terms in the span of vectors } \\ v_{\phi'} \text{ with } \phi'\in \Xi_{e,e'} \text{ for } e'<e} = 0
\end{equation}
for all $r\in\mathbb Z$. As before, we consider the cases $s\ge \lambda(h_{\mu,0})$ and $s< 0$ separately and prove the statement by a further induction on $s$ and $|s|$, respectively. If $s\ge\lambda(h_{\mu,0})$, set $r=s-1-\lambda(h_{\mu,0})$ in \eqref{1e>1} and notice that, for $1\le j\le \lambda(h_{\mu,0})$, we have  $0<  r+j+1 \le s$, from the induction easily follows.
For $s<0$, observe that $\Lambda_{\mu,l-e\lambda({h_{\mu,0}})}v=v \ne 0$ and, therefore,
$$\{j\in\mathbb Z_+:j\le\lambda(h_{\mu,0}) \text{ and }\Lambda_{\mu,l-ej}v \ne 0\}\ne\emptyset.$$
Let $j_0$ be the minimum of this set. Then, \eqref{1e>1} can be rewritten as
\begin{equation}\label{1e>1'}
\tsum_{j=j_0}^{\lambda(h_{\mu,0})} (x_{\mu,-(r+j+1)}^- \otimes t^{r+j+1})^{(e)}\Lambda_{\mu,l-ej}v + \substack{\text{ other terms in the span of vectors } \\ v_{\phi'} \text{ with } \phi'\in \Xi_{e,e'} \text{ for } e'<e} = 0.
\end{equation}
Since $\omega_{\mu,l-ej_0} \ne 0$ by definition of $j_0$, the induction on $|s|$ is easily done by using \eqref{1e>1'} with  $r=s-1-j_0$.
By using Lemma \ref{basicreltw}(b) and part (iii) of Lemma \ref{basicreltw}(c), respectively, the cases $\lie g$ is not of type $A_{2n}$ and $\mu \in R_l$, and $\lie g$ is of type $A_{2n}$ and $\mu \in 2R_s$ can be carried out similarly. It remains to consider the case $\lie g$ of type $A_{2n}$ and $\mu\in R_s$. This time we set $r=e$, $k=\lambda(h_{\mu,0})$, $s=r$ and $\alpha=\mu$ on part (iv) of Lemma \ref{basicreltw}(c) to get
\begin{equation}\label{2e>1}
\tsum_{j=1}^{\lambda(h_{\mu,0})+1} (x_{\mu,-(r+j-1)}^-\otimes t^{r+j-1})^{(e)}\Lambda_{\mu,ek+e-ej}v + \substack{\text{ other terms in the span of vectors } \\ v_{\phi'} \text{ with } \phi'\in \Xi_{e,e'} \text{ for } e'<e} = 0
\end{equation}
for all $r\in\mathbb Z$. Again, we proceed by induction on $s$ and $|s|$ separately working with \eqref{2e>1} similarly to the way we worked with \eqref{1e>1}. This completes the case $e=d$.

If $e<d$, we can assume  by inductions hypothesis on $d$, that $0\le s_j<\lambda(h_{\mu_j,0})$ for $j>1$ in \eqref{e:vphi}. Hence, an application of Lemma \ref{sbrw} completes the argument.
\endproof

\section{Twisted modules via restriction of the action}\label{s:tmvr}

We shall use the following notation and assumptions throughout the section. We assume that $m$ is not the characteristic of $\mathbb F$. Given a $U_\mathbb F(\tlie g)$-module $V$, we denote by $V^\sigma$ the module for $U_\mathbb F(\tlie g^\sigma)$ obtained by restriction of the action. We also let $\mathbb K$ be the fraction field of $\mathbb A$ (hence, $\mathbb K$ has characteristic zero).
The image of $a\in\mathbb A$ in $\mathbb F$ will be denoted by $\bar a$.

\subsection{Statement of the main theorem}

The goal of this section is to prove the following theorem.

\begin{thm}\label{t:tmvr}
Let $\gb\pi \in \cal P_\mathbb F^{\sigma,+}$. Then, there exists $\gb\omega\in\cal P_\mathbb F^+$ such that
$$V_\mathbb F(\gb \omega)^\sigma \cong V_\mathbb F(\gb\pi) \qquad\text{and}\qquad W_\mathbb F(\gb \omega)^\sigma \cong W_\mathbb F(\gb\pi).$$
 \end{thm}

\begin{rem}
In characteristic zero, Theorem \ref{t:tmvr} was proved in \cite{CFS}. Although our proof is parallel to the one from \cite{CFS}, part of the arguments used there, such as the explicit description of certain annihilating ideals, could not be replicated in positive characteristic. The technical differences are mostly concentrated in the results proved in Subsection \ref{s:surj} where the hypothesis that $m$ is not the characteristic of $\mathbb F$ is used. However, in the case of Weyl modules our alternate approach to annihilating ideals is not sufficient for proving the result intrinsically. Instead, we use the result in the characteristic zero setting and then use the technique of reduction modulo $p$.
\end{rem}

Before going into the technical details of the proof, let us construct the element $\gb\omega$ of the statement of Theorem \ref{t:tmvr}. Let
\begin{equation*}
\gb\pi=\tprod_{k=1}^\ell\tprod_{\epsilon=0}^{m-1}\gb\omega^\sigma_{\lambda_{k,\epsilon},\zeta^{m-\epsilon} a_k}
\end{equation*}
be a standard decomposition of $\gb\pi$  (see \eqref{standard}). By decomposing $\lambda_{k,\epsilon}$ as sum of fundamental weights (see \eqref{drinf}), it follows that
$$\gb \pi = \begin{cases} \prod_{k=1}^\ell \prod_{\epsilon=0}^{m-1}\prod_{i\in I_0} (\gb\omega_{\omega_{i},\zeta^{m-\epsilon}a_k}^\sigma)^{e_{k,\epsilon,i}}, & \text{ if $\lie g$ is not of type $A_{2n}$, } \\ \prod_{k=1}^\ell \prod_{\epsilon=0}^{m-1}\prod_{i\in I_0} (\gb\omega_{2^{\delta{i,j}}\omega_{i},\zeta^{m-\epsilon}a_k}^\sigma)^{e_{k,\epsilon,i}}, & \text{ if $\lie g$ is of type $A_{2n}$}, \end{cases}$$
for some $e_{k,\epsilon,i}$ and where $j$  in the second line is the unique index in $I_0$ such that $\alpha_j\in R_s$. Then,
\begin{equation}\label{e:tmvr}
\gb \omega = \tprod_{k=1}^\ell \gb \omega_{\mu_k,a_k} \qquad  \text{with} \qquad  \mu_k=\tsum_{\epsilon=0}^{m-1} \tsum_{i\in I_0}^{} e_{k,\epsilon,i}\sigma^\epsilon(\omega_{o(i)}).
\end{equation}

\subsection{On certain surjective homomorphisms}\label{s:surj}

Given a polynomial $g(u) = \prod_{j=1}^n (1-a_ju)\in\mathbb K[u]$, let
\begin{equation}\label{e:tilpoly}
\tilde g(t) = \prod_{j=1}^n (t-a_j)\in\mathbb K[t].
\end{equation}
For $\alpha\in R^+, k\in\mathbb Z$, define
\begin{equation*}
x_{\alpha,k, g}^\pm = x_\alpha^\pm\otimes t^k \tilde g(t).
\end{equation*}
Notice that, if $a_j\in\mathbb A$ for all $j$, then $(x_{\alpha,k, g}^\pm)^{(r)} \in U_\mathbb A(\tlie g)$ for all $r,k\in\mathbb Z,r\ge 0$. Moreover, one can easily check that the image $1\otimes (x_{\alpha,k, g}^\pm)^{(r)}$ of $(x_{\alpha,k, g}^\pm)^{(r)}$ in $U_\mathbb F(\tlie g)$ is nonzero.
The next lemma is trivially established.

\begin{lem}
Let $g(u) = \prod_{j=1}^n (1-a_ju),f(u) = \prod_{j=1}^n (1-b_ju)\in\mathbb K[u]$. Suppose that $a_j,b_j\in\mathbb A^\times$ and $\overline{a_j}=\overline{b_j}$ for all $j$. Then, $1\otimes (x_{\alpha,k, g}^\pm)^{(r)}=1\otimes (x_{\alpha,k, f}^\pm)^{(r)}$.\hfill\qedsymbol
\end{lem}

We can then define elements $(x_{\alpha,k, f}^\pm)^{(r)}\in U_\mathbb F(\tlie g)$ for any polynomial $f(u)=\prod_{j=1}^n (1-b_ju) \in\mathbb F[u]$
by setting
\begin{equation}\label{e:xkfinF}
(x_{\alpha,k, f}^\pm)^{(r)}=1\otimes (x_{\alpha,k,g}^\pm)^{(r)}
\end{equation}
where $g$ is any polynomial of the form $g(u) = \prod_{j=1}^n (1-a_ju)$ with $a_j\in\mathbb A^\times$ such that $\overline{a_j}=b_j$ for all $j$.

Given such $f\in\mathbb F[u]$, consider the ideal $I_\mathbb F^f$ of $U_\mathbb F(\tlie g)$ generated by $\{(x_{\alpha,k,f}^\pm)^{(r)}  :  \alpha\in R^+,r,k\in \mathbb Z,r>0\}$ and set $$U_\mathbb F(\tlie g)_f = \frac{U_\mathbb F (\tlie g)}{I_\mathbb F^f}.$$
Notice that $I_\mathbb F^f$ is a Hopf ideal of $U_\mathbb F(\tlie g)$ and, hence, the Hopf algebra structure on $U_\mathbb F(\tlie g)$ naturally induces one on $U_\mathbb F(\tlie g)_f$. Moreover, if $V,W$ are $U_\mathbb F(\tlie g)$-modules such that $I_\mathbb F^fV=I_\mathbb F^fW=0$, then $V\otimes W$ is naturally a module for $U_\mathbb F(\tlie g)_f$.

Let $\phi_f$ be the composition $U_\mathbb F(\tlie g^\sigma) \hookrightarrow U_\mathbb F(\tlie g) \to U_\mathbb F(\tlie g)_f$ where the last map is the canonical projection. The goal of this subsection is to prove the following proposition.

\begin{prop}\label{p:twistsurj}
Let $f(u)=\prod_{j=1}^n(1-b_ju)\in\mathbb F[u]$ and suppose that $b_i^m\ne b_j^m$ for all $i\ne j$. Then, for all $N\in \mathbb N$, the map $\phi_{f^N}$ is surjective.
\end{prop}

The following corollary is then easily deduced.

\begin{cor}\label{c:twistsurj}
Let $V$ be a $U_\mathbb F(\tlie g)$-module such that $I_\mathbb F^{f^N}V=0$ for some $N\in\mathbb N$ and suppose $S\subseteq V$ is such that $V=U_\mathbb F(\tlie g)S$. Then, $V=U_\mathbb F(\tlie g^\sigma)S$ and,  if $V$ is a simple $U_\mathbb F(\tlie g)$-module,  $V^\sigma$ is a simple $U_\mathbb F(\tlie g^\sigma)$-module.\hfill\qedsymbol
\end{cor}

Fix $g$ as in \eqref{e:xkfinF} and $N\in\mathbb N$, consider the ideal $I_\mathbb A^{g^N}$ of $U_\mathbb A(\tlie g)$ defined similarly, and let $U_\mathbb A(\tlie g)_{g^N} = \frac{U_\mathbb A (\tlie g)}{I_\mathbb A^{g^N}}$. Notice also that the image of $I_\mathbb A^{g^N}$ in $U_\mathbb F(\tlie g)$ is contained in $I_\mathbb F^{f^N}$ and, hence, we have an $\mathbb A$-linear epimorphism
\begin{equation*}
U_\mathbb A(\tlie g)_{g^N} \to U_\mathbb F(\tlie g)_{f^N}.
\end{equation*}
Moreover, the following diagram commutes
\begin{displaymath}
\xymatrix{ U_\mathbb A(\tlie g^\sigma) \ar[r] \ar[d] & U_\mathbb A(\tlie g) \ar[d] \ar[r] & U_\mathbb A(\tlie g)_{g^N} \ar[d]\\
U_\mathbb F(\tlie g^\sigma) \ar[r] & U_\mathbb F(\tlie g)  \ar[r] & U_\mathbb F(\tlie g)_{f^N} }
\end{displaymath}
Hence, in order to prove Proposition \ref{p:twistsurj}, it suffices to prove that the composition $\phi_{g^N}$ of the maps in the first row of the above diagram is surjective. Notice that the hypothesis $b_i^m\ne b_j^m$ for all $i\ne j$, together with the fact that $\mathbb A$ is a DVR with $\mathbb F$ as its residue field, implies that
\begin{equation}
a_i^m- a_j^m\in\mathbb A^\times \qquad\text{for all}\qquad i\ne j.
\end{equation}

Let $\bar x$ denote the image of $x\in U_\mathbb A(\tlie g)$ in $U_\mathbb A(\tlie g)_{g^N}$. Thus, we want to prove that
$\overline{(x_{\alpha}^\pm\otimes t^k)^{(r)}}$ belongs to the image of $\phi_{g^N}$ for all $\alpha\in R^+,r,k\in \mathbb Z, r\ge 1$. Henceforth we fix $\alpha\in R^+$. By \eqref{e:basisrel} we have
\begin{equation*}
(x_{\alpha}^\pm\otimes t^k)^{(r)}  = \left(\frac{1}{\Gamma_\alpha} \tsum_{\epsilon=0}^{m -1} x_{\alpha,\epsilon}^\pm \otimes t^k\right)^{(r)}.
\end{equation*}
Since $m$ is not the characteristic of $\mathbb F$ and $\mathbb A$ is a DVR with $\mathbb F$ as residue field, we have $\Gamma_\alpha \in \mathbb A^\times$.
If $\alpha+\sigma(\alpha)\notin R$, then $[x_{\alpha,\epsilon}^\pm \otimes t^k,x_{\sigma(\alpha),\epsilon}^\pm \otimes t^k]=0$ and it follows from \eqref{e:divpsum} that it suffices to show that $\overline{(x_{\alpha,\epsilon}^\pm \otimes t^k)^{(r)}}$ belongs to the image of $\phi_{g^N}$ for all $0\le \epsilon <m$, $r\ge 1$ and $k \in \mathbb Z$. If $\alpha+\sigma(\alpha)\in R$, notice that the subalgebra of $\tlie g$ generated by $\{x_{\alpha,0}^\pm\otimes t^k , x_{\alpha,1}^\pm\otimes t^k ,x_{\alpha+\sigma(\alpha),1}^\pm\otimes t^{2k}\}$ (with a fixed choice of $\pm$) is a Heisenberg subalgebra. Then, \eqref{e:heis}  implies that it suffices to show that $\overline{(x_{\alpha,\epsilon}^\pm \otimes t^k)^{(r)}}$ and $\overline{(x^\pm_{\alpha+\sigma(\alpha),1}\otimes t^{k})^{(r)}}$ belong to the image of $\phi_{g^N}$ for all $0\le \epsilon <m$, $r\ge 1$ and $k \in \mathbb Z$. This will follow from \eqref{cl3} and \eqref{cl4} below.

We claim that, given $\epsilon=0,\dots,m-1$, and $k\in \mathbb Z$, there exists $\Psi_k^\epsilon \in  t^{m-\epsilon} \mathbb A[t^m,t^{-m}]$ such that $\Psi_k^\epsilon-t^k$ is divisible by $\tilde{g}^N$ in $\mathbb A[t,t^{-1}]$ (where $\tilde g$ is given by \eqref{e:tilpoly}). Assuming this claim, we complete the proof as follows. Write $\Psi_k^\epsilon-t^k = \sum_{i=-n_1}^{n_2} e_it^i \tilde{g}^N$ for some $n_1,n_2\in \mathbb N$ and $e_i\in\mathbb A, -n_1\le i\le n_2$. Then, another application of \eqref{e:divpsum} shows that
\begin{equation}\label{e:xainIg}
(x_{\beta}^\pm \otimes (\Psi_k^\epsilon-t^k))^{(r)} \in I_\mathbb A^{g^N} \quad \text{ for all } \quad \beta\in R^+, r\in\mathbb N.
\end{equation}
It then follows from \eqref{e:basisrel}, \eqref{e:xainIg}, and either \eqref{e:divpsum} (if $\beta+\sigma(\beta)\notin R$) or \eqref{e:heis} applied to the Heisenberg  subalgebras spanned by $\{x_{\beta}^\pm\otimes (\Psi_k^\epsilon-t^k) , x_{\sigma(\beta)}^\pm\otimes (\Psi_k^\epsilon-t^k) ,x_{\beta+\sigma(\beta),1}^\pm\otimes (\Psi_k^\epsilon-t^k)^{2}\}$ (if $\beta+\sigma(\beta)\in R$) that
$$(x_{\beta,\epsilon}^\pm \otimes (\Psi_k^\epsilon-t^k))^{(r)}\in I_\mathbb A^{g^N} \quad \text{ for all } \quad \beta\in R^+, r\in\mathbb N.$$
Using this, we now show that
\begin{equation}\label{cl1}
    \overline{(x_{\alpha,\epsilon}^\pm \otimes t^k)^{(r)}} = \overline{(x_{\alpha,\epsilon}^\pm \otimes \Psi_k^\epsilon)^{(r)}}
     \quad \text{ for all } \quad 0\le \epsilon < m, \ r,k\in \mathbb Z,  r>0,
\end{equation}
and, if $\alpha+\sigma(\alpha)\in R$, that
\begin{equation}\label{cl2}
    \overline{(x_{\alpha+\sigma(\alpha),1}^\pm \otimes t^k)^{(r)}} = \overline{(x_{\alpha+\sigma(\alpha),1}^\pm \otimes \Psi_k^1)^{(r)}}
    \quad \text{ for all } \quad r,k\in \mathbb Z,  r>0.
\end{equation}
Indeed, since $x_{\alpha,\epsilon}^\pm \otimes t^k$ and $x_{\alpha,\epsilon}^\pm \otimes (t^k- \Psi_k^\epsilon)$ commute, we get from \eqref{e:divpsum} that
$$(x_{\alpha,\epsilon}^\pm \otimes \Psi_k^\epsilon)^{(r)}=(x_{\alpha,\epsilon}^\pm \otimes (t^k-(t^k-\Psi_k^\epsilon)))^{(r)} = \tsum_{j=0}^{r} (-1)^{j}(x_{\alpha,\epsilon}^\pm \otimes t^k)^{(r-j)}(x_{\alpha,\epsilon}^\pm \otimes (t^k-\Psi_k^\epsilon))^{(j)}.$$
Since the only summand with nonzero image in $U_\mathbb A(\tlie g)_{g^N}$ is the one with $j=0$,  \eqref{cl1} follows. Equation \eqref{cl2} is proved similarly using that $x_{\alpha+\sigma(\alpha),1}^\pm \otimes t^k$ and $x_{\alpha+\sigma(\alpha),1}^\pm \otimes (t^k- \Psi_k^1)$ commute.

Finally, let $\mu \in \wt(\lie g_\epsilon)\cap Q_0^+$ such that $\mu=\alpha|_{\lie h_0}$.  If  $\alpha \in O$, then $x_{\alpha,\epsilon}^\pm = x_{\mu,\epsilon}^\pm$ by \eqref{e:xmuxalpha}. Otherwise, we have $\sigma^j(\alpha) \in O$ for some $j=0,\cdots,m-1$, and, hence, $x_{\alpha,\epsilon}^\pm = \zeta^{j\epsilon} x_{\sigma^j(\alpha),\epsilon}^\pm=\zeta^{j\epsilon} x_{\mu,\epsilon}^\pm$ by \eqref{e:basisrel} (recall that $\zeta \in \mathbb A^\times$). Since $(x_{\mu,\epsilon}^\pm \otimes \Psi_k^\epsilon)^{(r)} \in U_\mathbb A(\tlie g^\sigma)$, it follows from \eqref{cl1} that
\begin{equation} \label{cl3}
\overline{(x_{\alpha,\epsilon}^\pm \otimes t^k)^{(r)}} = \phi_{g^N}\left(c_\alpha^r (x_{\mu,\epsilon}^\pm \otimes \Psi_k^\epsilon)^{(r)}\right)
\end{equation}
with $c_\alpha=1$ if $\alpha\in O$ and $c_\alpha=\zeta^{j\epsilon}$ if $\sigma^j(\alpha)\in O$ for some $j\ne 0$.  Similarly, if $\alpha+\sigma(\alpha)\in R^+$, then $(x_{2\mu,1}^\pm \otimes \Psi_k^1)^{(r)} \in U_\mathbb A(\tlie g^\sigma)$ and \eqref{cl2} implies that
\begin{equation}\label{cl4}
    \overline{(x_{\alpha+\sigma(\alpha),1}^\pm \otimes t^k)^{(r)}}=\phi_{g^N}\left((x_{2\mu,1}^\pm \otimes \Psi_k^1)^{(r)}\right).
\end{equation}

It remains to prove the existence of $\Psi_k^\epsilon$ as claimed.
Let $\langle \tilde{g}^N \rangle$ be the ideal of $\mathbb A [t,t^{-1}]$ generated by $\tilde{g}^N$. It suffices to show that the map
$$\varphi: \mathbb A [t^m,t^{-m}] \to \mathbb A [t,t^{-1}]/\langle \tilde{g}^N \rangle \quad \text{ given by } \quad t^k\mapsto \overline{t^k} \quad \text{ for all } \quad k\in m\mathbb Z$$ is surjective. Indeed, using this fact, let $\varphi_k^\epsilon$ be any element in $\varphi^{-1}(t^{k-m+\epsilon})$ and set $\Psi_k^\epsilon(t) = t^{m-\epsilon}\varphi_k^\epsilon(t)$ which clearly has the claimed properties.

To prove that $\varphi$ is surjective, observe that, since $a_i^m \ne a_j^m$ for $i\ne j$, the Chinese Remainder Theorem implies that we have an isomorphism
\begin{equation}\label{e:crt}
\mathbb A [t,t^{-1}]/\langle \tilde{g}^N \rangle \simeq \tprod_i^{} \mathbb A[t,t^{-1}]/\langle (t-a_i)^N\rangle
\end{equation}
and an epimorphism
\begin{equation}\label{e:crtm}
\mathbb A [t^m,t^{-m}]\to \tprod_i^{} \mathbb A[t^m,t^{-m}]/\langle (t^m-a_i^m)^N\rangle.
\end{equation}
Evidently, for every $a\in\mathbb A^\times$, we also have isomorphisms
\begin{equation}\label{e:no-}
\mathbb A[t^m,t^{-m}]/\langle (t^m-a^m)^N\rangle \simeq\mathbb A[t^m]/\langle (t^m-a^m)^N\rangle \quad \text{and} \quad   \mathbb A[t,t^{-1}]/\langle (t-a)^N\rangle \simeq \mathbb A[t]/\langle (t-a)^N\rangle.
\end{equation}
Note that $(t^m-a^m)^N \in \langle (t-a)^N \rangle$. Indeed, recalling that $m\in\{2,3\}$, we have
\begin{equation}\label{e:surjclaim}
t^m-a^m = (t-a)\left( ma^{m-1} + (t-a)(t+2a)^{\delta_{m,3}}\right).
\end{equation}
Hence, the surjectivity of $\varphi$ follows from \eqref{e:crt}, \eqref{e:crtm}, and \eqref{e:no-} if we prove that
$$\psi: \mathbb A[t^m]\to \mathbb A[t]/\langle(t-a)^N\rangle \quad \text{ given by } \quad t^{mk} \mapsto \overline{t^{mk}} \quad \text{ for all } \quad k\in \mathbb N$$ is surjective for all $a\in\mathbb A^\times$.
In order to prove that $\psi$ is surjective, let
$$\lie m=(t-a)\mathbb B \qquad\text{where}\qquad \mathbb B:=\mathbb A[t]/\langle(t-a)^N\rangle.$$
Using that $t(t-a)=a(t-a)+(t-a)^2$, one easily sees by induction on $n\ge 1$ that
\begin{equation}\label{e:surjclaim0}
t^n(t-a) \in (t-a)\mathbb A+ (t-a)^2\mathbb A[t] \quad\text{for all}\quad n\ge 1.
\end{equation}
As $t-a=\frac{(t^m-a^m)}{ma^{m-1}}-\frac{(t-a)^2}{ma^{m-1}}(t+2a)^{\delta_{m,3}}$ by \eqref{e:surjclaim} and recalling that $m,a\in\mathbb A^\times$, \eqref{e:surjclaim0} implies that
\begin{equation}\label{e:surjclaim1}
\lie m \subseteq \overline{(t^m-a^m)}\mathbb A  + \lie m^2.
\end{equation}
Using this and that $t^m-a^m\in(t-a)\mathbb A[t]$, one easily proves by induction on $n\ge 1$ that
\begin{equation}\label{e:surjclaim2}
\lie m^n\subseteq \overline{(t^m-a^m)^n}\mathbb A +\lie m^{n+1} \quad \text{ for all } \quad n\ge 1.
\end{equation}
Plugging \eqref{e:surjclaim2} back in \eqref{e:surjclaim1} and proceeding recursively, it follows that
\begin{equation}
\lie m\subseteq \tsum_{i=1}^n\overline{(t^m-a^m)^i}\mathbb A + \lie m^{n+1}\quad \text{ for all } \quad n\ge 1.
\end{equation}
Taking $n=N-1$ above and observing that $\overline{(t^m-a^m)^i}\mathbb A\in{\rm Im}\ \psi$, it follows that $\lie m\subseteq {\rm Im}\ \psi$.
On the other hand, since $\mathbb A\subseteq {\rm Im}\ \psi$ and $\mathbb B= \mathbb A+\lie m$, we conclude that $\psi$ is surjective as desired.

\subsection{Proof of Theorem \ref{t:tmvr} for simple modules}

Given $\gb\omega=\prod_{j=1}^n \gb\omega_{\lambda_j,a_j}\in\cal P_\mathbb F^+$ with $a_i\ne a_j$ for $i\ne j$, set
\begin{equation}\label{e:fomega}
f_{_\gb\omega}(u) = \tprod_{j=1}^n (1-a_ju) \qquad\text{and}\qquad I_\mathbb F^{\gb\omega }= I_\mathbb F^{f_{_\gb\omega}}.
\end{equation}

\begin{lem}\label{l:If} $I_\mathbb F^{\gb\omega} V_\mathbb F(\gb\omega)=0$.
\end{lem}

\proof
To shorten notation, set $f=f_\gb\omega$. Assume first that $\gb\omega = \gb\omega_{\lambda,a}$ for some $\lambda\in P^+$ and $a\in\mathbb F^\times$. In particular, $V_\mathbb F(\gb\omega)\cong V_\mathbb F(\lambda,a)$. Then, it is immediate from the definition of the evaluation map that, if $g\in\mathbb F[u]$ is a multiple of $f$ and has constant term $1$, then
\begin{equation}\label{e:If}
(x_{\alpha,k,g}^\pm)^{(r)}V_\mathbb F(\lambda,a) = (a^{k}\tilde g(a))^r(x_{\alpha}^\pm)^{(r)}V_\mathbb F(\lambda,a) = 0 \quad\text{for all}\quad \alpha\in R^+, k,r\in\mathbb Z,r>0,
\end{equation}
since $\tilde g(a)=0$.

For general $\gb\omega\in\cal P_\mathbb F^+$, say $\gb\omega=\prod_{j=1}^n \gb\omega_{\lambda_j,a_j}$ with $a_i\ne a_j$ for $i\ne j$,
we have $V_\mathbb F(\gb\omega)\cong \otm_{i=1}^n V_\mathbb F(\lambda_j,a_j)$ by Proposition \ref{p:evmoduntw}. Then, given $v=v_1\otimes\cdots\otimes v_n$ with $v_j\in V_\mathbb F(\lambda_j,a_j)$, it follows from \eqref{e:If} and \eqref{e:comultdp} with $x=x_{\alpha,k,f}$ that
\begin{equation*}
(x_{\alpha,k,f})^{(r)} v = 0 \quad\text{for all} \quad \alpha\in R^+, k,r\in\mathbb Z,r>0,
\end{equation*}
which implies $I_\mathbb F^\gb\omega V_\mathbb F(\gb\omega)=0$.
\endproof

We are ready to prove Theorem \ref{t:tmvr} in the case of simple modules. Recall the definition of $\gb\omega$ in \eqref{e:tmvr} and that $a_i^m \ne a_j^m$ for $i\ne j$ by \eqref{standard}. It then follows from Lemma \ref{l:If} and Corollary \ref{c:twistsurj} that $V_\mathbb F(\gb \omega)^\sigma$ is simple. Evidently, if $v$ is a highest-$\ell$-weight vector of $V_\mathbb F(\gb \omega)$, then $U_\mathbb F((\tlie n^+)^\sigma)^0v=0$. Hence, it remains to prove that
\begin{equation*}
\Lambda_i^{\sigma,+}(u)\ v = \gb\pi_i(u) \ v \quad\text{for all}\quad i\in I_0.
\end{equation*}

Suppose first that $o(i)$ is not fixed by $\sigma$. Then,
\begin{align*}
\Lambda_i^{\sigma,+}(u) v  &= \tprod_{j=0}^{m-1} \Lambda_{\sigma^j(o(i))} (\zeta^{m-j}u) v = \tprod_{j=0}^{m-1} \gb\omega_{\sigma^j(o(i))}(\zeta^{m-j}u) v  \\ &=\tprod_{k=1}^\ell \tprod_{j=0}^{m-1} (\gb\omega_{\sigma^j(o(i)),\zeta^{m-j}a_k})^{\mu_k(h_{\sigma^j(o(i))})}_{\sigma^j(o(i))} v \\
&\stackrel{*}{=}\tprod_{k=1}^\ell \tprod_{j=0}^{m-1} (1-\zeta^{m-j} a_k)^{e_{k,j,i}} v = \gb\pi_i(u)v,
\end{align*}
where equality $*$ follows from the definition of $\mu_k$ in \eqref{e:tmvr} and the last equality follows from \eqref{drinf}.  Similarly, if $o(i)$ is fixed by $\sigma$, then
\begin{align*}
\Lambda_i^{\sigma,+}(u)\ v  &= \Lambda_{o(i);m}^{+}(u)\ v = \tprod_{k=1}^\ell (1-a_k^mu)^{\mu_k(h(o(i)))} \\
 & = \tprod_{k=1}^\ell \tprod_{\epsilon=0}^{m-1} (1-a_k^mu)^{e_{k,\epsilon,i}} v \\
 & = \tprod_{k=1}^\ell \tprod_{\epsilon=0}^{m-1} (1-(\zeta^{m-\epsilon}a_k)^mu)^{e_{k,\epsilon,i}} v = \gb\pi_i(u)v,
\end{align*}
since $(\zeta^{m-\epsilon})^m=1$ for all $0 \le \epsilon < m-1$ and the last equality once again follows from \eqref{drinf}.\hfill\qedsymbol

The next result is immediate from Proposition \ref{p:evmoduntw} and the above proof of Theorem \ref{t:tmvr} and recovers the classical result that the simple modules are tensor products of evaluation modules.

\begin{cor}
$V_\mathbb F(\gb \pi )$ is isomorphic to $\otm_{k=1}^\ell V_\mathbb F(\mu_k,a_k)^\sigma$.\hfill\qedsymbol
\end{cor}

\begin{rem}
Combining the last corollary with \cite[Theorem 3.4]{hyperlar}, one gets a version of Steinberg's tensor product theorem for twisted hyper loop algebras.
\end{rem}

\subsection{Proof of Theorem \ref{t:tmvr} for Weyl modules}

Let  $\cal P_\mathbb A^\times$ be the subset of $\cal P_\mathbb K^+$ consisting of $I$-tuples of polynomials whose roots lie in $\mathbb A^\times$. Similarly, define the subset $\cal P_\mathbb A^{\sigma,\times}$ of $\cal P_\mathbb K^{\sigma,+}$.

\begin{thm}
Let $\gb\varpi\in\cal P_\mathbb A^\times$ (respectively, $\gb\varpi\in\cal P_\mathbb A^{\sigma,\times}$), $v$ a highest-$\ell$-weight vector of $W_\mathbb K(\gb\varpi)$, and $L(\gb\varpi)=U_\mathbb A(\tlie g)v$ (respectively, $L(\gb\varpi)=U_\mathbb A(\tlie g^\sigma)v$). Then, $L$ is a free $\mathbb A$-module and the canonical map $\mathbb K\otimes_\mathbb A L(\gb\varpi)\to W_\mathbb K(\gb\varpi)$ is an isomorphism.
\end{thm}

\begin{proof}
The non-twisted statement was proved in \cite[Theorem 4.5(b)]{hyperlar} and the twisted version is proved similarly by repeating the steps of the proof of Theorem \ref{t:twfd} (cf. proof of \cite[Theorem 4.2]{hyperlar}).
\end{proof}

Let $L(\gb\varpi)$ be as in the last theorem and set
\begin{equation*}
L_\mathbb F(\gb\varpi) = \mathbb F\otimes_\mathbb A L(\gb\varpi).
\end{equation*}
Then, $L_\mathbb F(\gb\varpi)$ is a $U_\mathbb F(\tlie g)$-module (respectively, a $U_\mathbb F(\tlie g^\sigma)$-module) which is clearly a quotient of $W_\mathbb F(\gb\lambda)$ where $\gb\lambda$ is the image of $\gb\varpi$ in $\cal P_\mathbb F^+$ (respectively, in $\cal P_\mathbb F^{\sigma,+}$).

\begin{conj}\label{conj}
$L_\mathbb F(\gb\varpi)\cong W_\mathbb F(\gb\lambda)$. In particular, $\dim W_\mathbb F(\gb\lambda) = \dim W_\mathbb K(\gb\mu)$ for any $\gb\mu\in\cal P_\mathbb K^+$ (respectively, $\gb\mu\in\cal P_\mathbb K^{\sigma,+}$) such that $\wt(\gb\mu)=\wt(\gb\lambda)$.
\end{conj}

\begin{rem}\label{r:conj}
In the non-twisted case, the above conjecture was formulated in \cite{hyperlar}. It can be regarded as an analogue of a conjecture of Chari and Pressley \cite{CPweyl} saying that all Weyl modules for $U_\mathbb C(\tlie g)$ can be obtained as classical limits of appropriate Weyl modules for quantum affine algebras. Chari-Pressley's conjecture has been proved in \cite{chalok} (if $\lie g$ is of type $A$), in \cite{foli:weyldem} (for simply laced $\lie g$), and in \cite{naoi} (in general). A different approach has been pointed out by Nakajima using the global basis theory developed in \cite{bn,kascb,kaslz,nakq,nake} (\cite{naoi} also uses global basis results, but in a different manner). Nakajima's argument, which is outlined in the introduction of \cite{foli:weyldem}, most likely can be used to give a proof of Conjecture \ref{conj} as well (including the twisted version). See also \cite{bmm} for an approach for proving Conjecture \ref{conj} along the lines of \cite{foli:weyldem,naoi}.
\end{rem}

Recall the definition of $f_{_\gb\omega}$ from \eqref{e:fomega} and, given $N\in \mathbb N$, let
\begin{equation*}
I_\mathbb F^{\gb\omega,N} = I_\mathbb F^{f_{_\gb\omega}^N}.
\end{equation*}
Similarly, define $I_\mathbb K^{\gb\varpi,N}$ for $\gb\varpi\in\cal P_\mathbb K^+$ and $I_\mathbb A^{\gb\varpi,N}$ for $\gb\varpi\in\cal P_\mathbb A^\times$.
The following is the analogue of Proposition \ref{p:twistsurj} for Weyl modules.

\begin{lem}\label{l:IfN} For all $\gb\omega\in\cal P_\mathbb F^+$, there exists $N\in \mathbb Z_+$ such that $I_\mathbb F^{\gb\omega,N} W_\mathbb F(\gb\omega)=0$.
\end{lem}

\proof
Let $\gb\varpi\in\cal P_\mathbb A^\times$ be such that its image in $\cal P_\mathbb F^+$ is $\gb\omega$. It follows from \cite[Proposition 2.7]{CFS} that there exists $N\in \mathbb Z_+$ such that $I_{\mathbb K}^{\gb\varpi,N} W_{\mathbb K}(\gb\varpi)=0$. Hence, $I_{\mathbb A}^{\gb\varpi,N} L(\gb\varpi)=0$ which implies $\overline{I_{\mathbb A}^{\gb\varpi,N}} L_\mathbb F(\gb\varpi)=0$ where $\overline{I_{\mathbb A}^{\gb\varpi,N}}$ is the image of $I_{\mathbb A}^{\gb\varpi,N}$ in $U_\mathbb F(\tlie g)$. Since $I_\mathbb F^{\gb\omega,N}$ is clearly contained in $\overline{I_{\mathbb A}^{\gb\varpi,N}}$, it follows that
$I_{\mathbb F}^{\gb\omega,N} L_\mathbb F(\gb\varpi)=0$ and the lemma follows from Conjecture \ref{conj}.
\endproof

We are ready to prove Theorem \ref{t:tmvr} in the case of Weyl modules.  Let $\gb\omega$ be as in \eqref{e:tmvr}, $N$ as in Lemma \ref{l:IfN}, and fix a highest-$\ell$-weight vector $v$ for $W_\mathbb F(\gb\omega)$. The same computation closing the proof of Theorem \ref{t:tmvr} for simple modules shows that $U_{\mathbb F}(\tlie g^\sigma)v$ is a quotient of $W_\mathbb F(\gb\pi)$. Moreover, it follows from Lemma \ref{l:IfN} and Corollary \ref{c:twistsurj} that
\begin{equation*}
W_\mathbb F(\gb\omega) = U_{\mathbb F}(\tlie g^\sigma)v.
\end{equation*}

On the other hand, let $\gb\varpi\in\cal P_\mathbb A^{\sigma,\times}$ be such that its image in $\cal P_\mathbb F^{\sigma,+}$ is $\gb\pi$. Then, by the twisted part of Conjecture \ref{conj}, $\dim W_\mathbb F(\gb\pi) = \dim W_\mathbb K(\gb\varpi)$. By the characteristic zero version of Theorem \ref{t:tmvr} proved in \cite{CFS}, we have $\dim W_\mathbb K(\gb\varpi)=\dim W_\mathbb K(\gb\mu)$ for some $\gb\mu\in\cal P_\mathbb K^+$ such that $\wt(\gb\mu)=\wt(\gb\omega)$ (in fact, it is easy to see that $\gb\mu\in\cal P_\mathbb A^\times$ and that its image in $\cal P_\mathbb F^+$ is $\gb\omega$). Another application of the non-twisted part of Conjecture \ref{conj} completes the proof.\hfill\qedsymbol

\begin{rem}
Using Theorem \ref{t:tmvr}, one can show by the same arguments used in characteristic zero (see \cite{S2}) that the blocks of the category of finite-dimensional $U_\mathbb F(\tlie g^\sigma)$-modules are described as in the characteristic zero setting.
\end{rem}


\begin{thebibliography}{}
\bibitem{bn} J.~Beck and H.~Nakajima, {\em Crystal bases and two-sided cells of quantum affine algebras}, Duke Math. J. {\bf 123} (2004), 335--402.
\bibitem{bia:phd} A. Bianchi, Representações de hiperálgebras de laços e álgebras de multicorrentes, Ph.D. Thesis, Unicamp (2012).
\bibitem{bmm} A. Bianchi, A. Moura, and T. Macedo, {\em On Demazure and local Weyl modules for affine hyperalgebras}, work in progress.
\bibitem {CFS} V. Chari, G. Fourier, and P. Senesi, {\em Weyl Modules for the twisted loop algebras},  J. Algebra {\bf 319} (2008), no. 12, 5016--5038.
\bibitem{chalok} V. Chari and S. Loktev, {\em Weyl, Demazure and fusion modules for the current algebra of $\lie{sl}_{r+1}$}, Adv. in Math., {\bf 207} (2006), no. 2, 928--960.
\bibitem{CPweyl} V. Chari and  A. Pressley, {\em Weyl modules for classical and quantum affine algebras}, Represent. Theory  {\bf 5}  (2001), 191--223.
\bibitem{Ido} I. Efrat,  Valuation, Orderings, and Milnor $K$-Theory, Mathematical Surveys and Monographs AMS {\bf 124} (2006).
\bibitem{fkks} G. Fourier, T. Khandai, D. Kus, and A. Savage, {\em Local Weyl modules for equivariant map algebras with free abelian group actions}, J. Algebra {\bf 350} (2012), 386--404, DOI:10.1016/j.jalgebra.2011.10.018.
\bibitem{fk} G. Fourier and D. Kus, {\em Demazure modules and Weyl modules: The twisted current case}, to appear in Trans. Amer. Math. Soc., arXiv:1108.5960.
\bibitem{foli:weyldem} G. Fourier and P. Littelmann, {\em Weyl modules, Demazure modules, KR-modules, crystals, fusion products and limit constructions}, Adv. in Math. {\bf 211} (2007), no. 2, 566--593.
\bibitem{fms} G. Fourier, N. Manning, and P. Senesi, {\em Global Weyl modules for the twisted loop algebra}, arXiv:1110.2752.
\bibitem{G} H. Garland, {\em The arithmetic theory of loop algebras}, J. Algebra {\bf 53} (1978), 480--551.
\bibitem{H} J. E. Humphreys, {\em Introduction to Lie algebras and representation theory}, Springer--Verlag, GTM {\bf 9} (1970).
\bibitem{hyperlar} D. Jakelic and A. Moura, {\em Finite-dimensional representations of hyper loop algebras}, Pacific J. Math. {\bf 233} (2007), no. 2, 371 - 402.
\bibitem{JM2} D. Jakelic and A. Moura, {\em Finite-dimensional representations of hyper loop algebras over non-algebraically closed fields}, Algebras and Representation Theory {\bf 13} (2010), no. 3, 271--301.
\bibitem{JM3} D. Jakelic and A. Moura, {\em On multiplicity problems for finite-dimensional representations of hyper loop algebras}, Contemp. Math. {\bf 483} (2009), 147--159.
\bibitem{janb} J. Jantzen, {\em Representations of algebraic groups}, Academic Press (1987).
\bibitem{Kac} V. Kac, {\em Infinite dimensional Lie algebras}, Cambridge University Press, New York, 1990.
\bibitem{kascb} M.~Kashiwara, {\em Crystal bases of modified quantized enveloping algebras}, Duke Math. J. {\bf 73}  (1994), 383--413.
\bibitem{kaslz} M.~Kashiwara, {\em On level zero representations of quantized affine algebras}, Duke Math. J. {\bf 112}  (2002), 117--195.
\bibitem{kosagz} B. Kostant, {\em Groups over $\mathbb Z$}, Algebraic Groups and Discontinuous Subgroups, Proc. Symp. Pure Math. IX, Providence, AMS (1966).
\bibitem{mitz} D. Mitzman, {\em Integral Bases for Affine Lie Algebras and Their Universal Enveloping Algebras}, Contemp. Math. {\bf 40} (1983).
\bibitem{nakq} H.~Nakajima, {\em Quiver varieties and finite-dimensional representations of quantum affine algebras}, J. Amer. Math. Soc. {\bf 14} (2001),
145--238.
\bibitem{nake} H.~Nakajima, {\em Extremal weight modules of quantum affine algebras}, Adv. Stud. Pure Math. {\bf 40} (2004),  343--369.
\bibitem{naoi} K. Naoi, {\em Weyl modules, Demazure modules and finite crystals for non-simply laced type}, Adv. in Math. {\bf 229} (2012), no. 2, 875--934.
\bibitem{ns} E. Neher and A. Savage, {\em Extensions and block decompositions for finite-dimensional representations of equivariant map algebras}, arXiv:1103.4367.
\bibitem{nss} E. Neher, A. Savage, and P. Senesi, {\em Irreducible finite-dimensional representations of equivariant map algebras}, Trans. Amer. Math. Soc. {\bf 364} (2012), no. 5, 2619--2646, DOI: 10.1090/S0002-9947-2011-05420-6.
\bibitem{shari} S. Prevost. {\em Vertex Algebras and Integral Bases for the Enveloping Algebras of Affine Lie Algebras}, Mem. Math. Amer. Soc. {\bf 466} (1992).
\bibitem{S2} P. Senesi, {\em  The block decomposition of finite-dimensional representations of twisted loop algebras}, Pacific J. Math. {\bf 244} (2010), no. 2, 355--357.
\bibitem{serre} J.-P Serre, {\em Local Fields}, Springer-Verlag GTM {\bf 67} (1980).


\end{thebibliography}
\end{document}